\documentclass[11pt]{article}
\usepackage{amsmath, latexsym, amsfonts, amssymb, amsthm, amscd}
\usepackage{color}
\usepackage{dsfont}
\usepackage{xcolor}

\numberwithin{equation}{section}

\newtheorem{theorem}{Theorem}[section]
\newtheorem{corollary}[theorem]{Corollary}
\newtheorem{proposition}[theorem]{Proposition}
\newtheorem{lemma}[theorem]{Lemma}

\newtheorem{definition}{Definition}[section]

\newcommand{\ud}{\mathrm{d}} 
\renewcommand{\P}{\mathbb{P}}
\newcommand{\R}{\mathbb{R}}

\newcommand{\E}{\mathbb{E}}

\newcommand{\N}{\mathbb{N}}

\newcommand{\A}{\mathcal{A}}
\newcommand{\U}{\mathcal{U}}

\newcommand{\zmin}{\vartheta}
\newcommand{\func}{\Phi}
\newcommand{\funcG}{G}
\newcommand{\funck}{k}

\title{
\textbf{Extinction time of  logistic branching processes  in a Brownian  environment}}

\author{ H. Leman\footnote{ {\sc Universit\'e de Lyon, Inria, CNRS, ENS de Lyon, UMPA UMR 5669, 46 all\'ee d'Italie, 69364 Lyon, France}. E-mail: helene.leman@inria.fr} \,\, and J.C. Pardo\footnote{ {\sc Centro de Investigaci\'on en Matem\'aticas A.C. Calle Jalisco s/n. 36240 Guanajuato, M\'exico.} E-mail: jcpardo@cimat.mx. Corresponding author}\\}

\begin{document}

\maketitle
\begin{abstract}
In this paper, we study the  extinction time of logistic branching processes  which are perturbed by an independent random environment driven by   a Brownian motion.
 Our arguments  use  a Lamperti-type representation 
which is interesting on its own right and provides a one to one correspondence between the latter family of processes and the family of Feller diffusions which are perturbed by an independent spectrally positive L\'evy process. When the  independent random perturbation (of the Feller diffusion) is driven by a subordinator then the logistic branching processes in a Brownian environment converges to a specified distribution; otherwise, it becomes extinct a.s. In the latter scenario, and following a similar approach to Lambert \cite{Lambert2005}, we provide
the expectation and the Laplace transform of the absorption time,
as a functional of the solution to a Ricatti differential equation.  In particular, the latter characterises the law of the process coming down from infinity.
	\bigskip 
	
	
		\bigskip
		
\end{abstract}

\section{Introduction and main results.}
 The prototypical example of  continuous state branching processes (or CB-processes) with competition is the so-called logistic Feller diffusion 
   which is defined as the unique strong solution of the following stochastic differential equation (SDE),
\[
Y_t=Y_0+b\int_0^tY_s \ud s+\int_0^t\sqrt{2\gamma^2Y_s}\ud B^{(b)}_s-c\int_0^tY_s^2\ud s,\qquad t\geq0, 
\]
where $b\in \mathbb{R}, c>0$ and $B^{(b)}=(B^{(b)}_t; t\geq0)$ is a standard Brownian motion. Such family of processes and their extensions have been studied by several authors, see for instance Berestycki et al. \cite{BFF}, Foucart \cite{Fou}, Lambert \cite{Lambert2005}, Ma \cite{Ma2015}, Pardoux \cite{Pardoux}  and the references therein. An important feature of the  logistic Feller diffusion is that it can also be constructed as scaling limits of Bienaym\'e-Galton-Watson processes with competition which are continuous time Markov chains  where individuals behave independently from one another and each giving birth to a random number of offspring 
 but also considering competition pressure. In other words, each pair of individuals interact at a fixed rate and one of them is killed as result of such interaction.  For further details of such convergence, we refer to Section 2.4 in Lambert \cite{Lambert2005}.
 
 Using a Lamperti-type  representation (random time change), Lambert \cite{Lambert2005} generalised the logistic Feller diffusion by replacing the  diffusion term with a general  continuous state branching process (CB-process for short). More precisely,  Lambert considered the following generalised Ornstein-Uhlenbeck process  starting from $x>0$, which is described as the unique strong solution of 
\[
\ud R_t=\ud X_t-cR_t\ud t,
\]
where  $c>0$ and $X=(X_t, t\ge 0)$ denotes a spectrally positive L\'evy process, that is to say, a c\`adl\`ag stochastic process with independent and stationary increments with no negative jumps. We denote by $\mathbf{P}_x$ for the law  of $X$  started from $x\in \mathbb{R}$ and for simplicity, we let $\mathbf{P}=\mathbf{P}_0$. 

It is known that the law of any spectrally positive L\'evy process $ X$ is completely characterized  by its Laplace exponent $\psi$ which is defined as $\psi(\lambda)=\log \mathbf{E}[e^{-\lambda X_1} ]$, for $\lambda\ge 0$,  and satisfies the so-called L\'evy-Khintchine formula
\begin{equation}\label{LevyKinthcine}
\psi(\lambda)=-b\lambda+\gamma^2 \lambda^{2}+\int_{(0,\infty)}\Big(e^{-\lambda u}-1+\lambda u\mathbf{1}_{\{u<1\}}\Big)\mu(\ud u),
\end{equation}
where $b\in \mathbb{R}$, $\gamma\ge 0$ and $\mu$ is a Radon measure concentrated on $(0,\infty)$ satisfying
\begin{equation}\label{mugral}
\int_{(0, \infty)}(1\land u^2)\mu(\ud u)<\infty.
\end{equation}
It is also known that the triplet $(b, \gamma, \mu)$ characterises the law of $X$. According to Theorem 17.5 in Sato \cite{Sa}, the following log-moment condition
\[
\mathbf{E}\Big[\log^+ X_1\Big]<\infty,
\]
is necessary and sufficient for the process $R$ to possess an invariant distribution.  From Theorem 25.3 in \cite{Sa}, the previous log-moment condition  is equivalent to 
 \begin{equation}\label{logcondition}
 \int_1^\infty \log(u)\mu(\ud u)<\infty.
 \end{equation}
For further details on L\'evy   and generalised Ornstein-Uhlenbeck processes, we refer to the monograph of  Sato \cite{Sa}.

Let $T^R_0$ denotes the first hitting time of 0 of  the generalised Ornstein-Uhlenbeck process $R$, i.e. 
$T_0^R:=\inf\{s:R_s =0\},$
and  consider the random clock
\[
\eta_t=\int_0^{t\land T^R_0}\frac{\ud s}{R_s}, \qquad \textrm{for} \quad t>0.
\]
Let $C$ denotes the right-continuous inverse of the clock $\eta$. According to Lambert \cite{Lambert2005},  the logistic branching process is defined as follows
\[
Y_t=\left\{ \begin{array}{ll}
R_{C_t} & \textrm{ if } 0\le t<\eta_\infty\\
0 & \textrm{ if } \eta_\infty<\infty  \textrm{ and } t\ge \eta_\infty.
\end{array}
\right .
\]
As it was observed by Foucart \cite{Fou},  the above definition is inconsistent with the fact that the process $R$ is  positive, drifts to $\infty$ and $\eta_\infty<\infty$, a.s. The latter may occur when 
\[
\mathcal{E}:=\int_0^\theta \frac{1}{x}\exp\left\{\frac{2}{c}\int_x^\theta\frac{\psi(u)}{u}\ud u\right\}\ud x<\infty,\qquad \textrm{for some} \qquad \theta>0,
\]
according to Lemma 4 in \cite{Fou}. Actually, the later integral condition is  necessary and sufficient   for the logistic branching process $Y$ to explode with positive probability. We also point out that  the process $Y$ does not explode a.s., if the log-moment condition \eqref{logcondition}  holds since it implies that $\mathcal{E}=\infty$. The latter follows from the fact that 
\[
\int_{0+} \frac{\psi(z)}{z}\ud z <\infty \qquad\textrm{is equivalent to }\qquad \int^\infty \log(u)\mu(\ud u)<\infty,
\]
see  for instance Corollary 3.21 in Li~\cite{li2010}.

In  \cite{Fou}, the author is interested in  studying the long term behaviour of the extension of the logistic branching process $Y$ on $[0,\infty]$ where the state $\infty$ might be an entrance, reflecting or an exit boundary. In particular, Foucart improved the results of Lambert \cite{Lambert2005} for such extension. In this paper, we are not interested in the  extension of $Y$, so that we continue our exposition below in the setting of \cite{Lambert2005}.

  When $c=0$, the process $Y$ is the so-called   CB-process
and the previous random time change relationship is known as the   Lamperti transform which was established by Lamperti \cite{La}.  In other words, a CB-process is associated with a spectrally positive L\'evy process  and in particular with its Laplace exponent $\psi$ which takes the role of the offspring  generating function in the compound Poisson case. Formally speaking, we  shall refer to  all $\psi$  which respect the definition \eqref{LevyKinthcine} as branching mechanisms.

 Interesting path properties of the logistic branching processes were derived by Lambert  \cite{Lambert2005} as consequence of this path transformation. 
For instance, in the case when the process $X$ is a subordinator, i.e. the branching mechanism is of the form
\[
\psi(z)=-\delta z-\int_{(0,\infty)}(1-e^{-zu})\mu(\ud u), \qquad z\geq 0,
\]
with $\delta\ge 0$, satisfying the  log-moment condition \eqref{logcondition}  and one of the following conditions: either $\delta \ne 0$, $\mu(0, \infty)=\infty$ or $c<\mu(0, \infty)<\infty$,   then the process $Y$ is positive recurrent on $(\delta/c, \infty)$ and possesses a stationary distribution which can be computed explicitly. Moreover if \eqref{logcondition} holds  but none of the latter conditions are satisfied, then the process $Y$ is null recurrent in $(0,\infty)$ and converges to $0$ in probability (see Theorem 3.4 in \cite{Lambert2005}). 

When $X$ is not a subordinator and condition \eqref{logcondition} is satisfied, then the process $Y$ goes to $0$ a.s. Moreover,  the process $Y$ gets extinct in finite time a.s.
accordingly as 
\begin{equation}\label{grey}
\int^\infty \frac{\ud z}{\psi(z)}<\infty,
\end{equation}
which is the so-called Grey's condition. Let $T^{Y}_0$ denotes the time to extinction  of the process $Y$, i.e $T_0^Y:=\inf\{t\ge 0: Y_t=0\}$. In \cite{Lambert2005}, under Grey's condition, the  Laplace transform of $T^Y_0$  was computed explicitly  and   the law of the process coming down from infinity was also determined.

It is important to note that the logistic branching process $Y$ can also be defined (up to time to explosion) as the unique strong solution of a SDE  which can also  be extended to more general competition mechanisms. To be more precise, let us consider  a general   competition mechanism $g$ which is a non-decreasing continuous function  on $[0,\infty)$ with $g(0)=0$, hence  the  branching process with  competition  satisfies the following SDE 
\[
\begin{split}
Y_t&=Y_0+\int_0^t bY_s \ud s-\int_0^tg(Y_s)\ud s+\int_0^t\sqrt{2\gamma^2Y_s}\ud B^{(b)}_s\\
&\hspace{2cm}+\int_0^t\int_{(1,\infty)}\int_0^{Y_{s-}} z{N}^{(b)}(\ud s, \ud z, \ud u)+\int_0^t\int_{(0,1)}\int_0^{Y_{s-}} z\widetilde{N}^{(b)}(\ud s, \ud z, \ud u),
\end{split}
\]
up to  explosion, where $B^{(b)}$ is a standard Brownian motion which is independent of the  Poisson random measure $N^{(b)}$ which is defined  on $\mathbb{R}_+^3$, with intensity measure $\ud s \mu(\ud z) \ud u$ such that  $\mu$ satisfies \eqref{mugral}
 and $\widetilde{N}^{(b)}$ denotes its compensated version.  Such SDE was considered by Ma \cite{Ma2015} (see also Berestycki et al. \cite{BFF}) in the particular case when 
 $\psi$ satisfies \eqref{LevyKinthcine} with 
\begin{equation}\label{integralcond1}
\int_{(0,\infty)}(u \land u^2)\mu(\ud u)<\infty,
\end{equation}
or equivalently $|\psi^\prime(0+)|<\infty$. Such assumption simplifies the previous SDE by modifying  the linear and the jump structure terms as follows
\[
\begin{split}
Y_t&=Y_0-\psi^\prime(0+)\int_0^t Y_s \ud s-\int_0^tg(Y_s)\ud s+\int_0^t\sqrt{2\gamma^2Y_s}\ud B^{(b)}_s+\int_0^t\int_{(0,\infty)}\int_0^{Y_{s-}} z\widetilde{N}^{(b)}(\ud s, \ud z, \ud u).
\end{split}
\]
Moreover, under condition \eqref{integralcond1}  the previous SDE does not explode a.s.

Our aim is to study the time to extinction of a generalized version of the logistic branching process which includes an extra randomness coming from an independent Brownian motion  which can be interpreted as a random environment. To be more precise, we consider 
the following SDE
\begin{equation}\label{SDEBL}
\begin{split}
Z_t&=Z_0+\int_0^t\Big(bZ_s -cZ^2_s\Big)\ud s+\int_0^t\sqrt{2\gamma^2Z_s}\ud B^{(b)}_s+\sigma\int_0^t Z_{s}\ud  B^{(e)}_s\\
&\hspace{3cm}+\int_0^t\int_{[1,\infty)} \int_0^{Z_{s-}}z{N}^{(b)}(\ud s, \ud z, \ud u) +\int_0^t\int_{(0,1)} \int_0^{Z_{s-}}z\widetilde{N}^{(b)}(\ud s, \ud z, \ud u),
\end{split}
\end{equation}
up to explosion, with $b$, $\gamma$, the Brownian motion $B^{(b)}$ and the Poisson random measure $N^{(b)}$ being as before and  where $c, \sigma \geq 0$ and $B^{(e)}$ is a standard Brownian motion independent of $B^{(b)}$ and $N^{(b)}$. The SDE \eqref{SDEBL} has a unique non-negative strong solution which satisfies the Markov property, see for instance  Theorem 1 in Palau and Pardo \cite{PP}.

When $c=0$, the family of processes described by \eqref{SDEBL} was introduced independently by He et al. \cite{he2016continuous} and by Palau and Pardo \cite{PP} with $B^{(e)}$ replaced by a L\'evy process  under the name of CB-processes   in a L\'evy random environment.  In this particular case,  the process $Z$  satisfies the branching property conditionally on the environment $B^{(e)}$ (quenched branching property). This particular case  (i.e. $c=0$) was studied by Palau and Pardo in \cite{PP0} where the probability of survival and non-explosion is explicitly determined when the branching mechanism is stable, i.e. $\psi(\lambda)=c_\alpha \lambda^\alpha$, for $\lambda>0$, with  $\alpha\in (0,1)\cup (1,2]$ and $c_\alpha<0$ or $c_\alpha>0$ accordingly as $\alpha\in (0,1)$ or $\alpha\in (1,2]$. The latter events can be computed in a closed-form in this  case, since the Laplace transform of $Z$ is explicit, a property which is derived from the quenched branching property of $Z$. We point out that in \cite{PP0}  there are not necessary and sufficient conditions for CB-processes in a Brownian random environment to explode or become extinct. Under  the finite moment condition \eqref{integralcond1},  CB-processes in a L\'evy random environment do not explode (see for instance Lemma 7 in Bansaye et al. \cite{BPS1}) and moreover, according to He et al. \cite{he2016continuous}  Grey's condition \eqref{grey} is a necessary and sufficient condition for the process
to become extinct  with positive probability, see Theorem 4.1 in   \cite{he2016continuous}. In the particular case when the random environment is driven by a Brownian motion with  drift,  the associated CB-process in random environment becomes extinct at finite time a.s. if the drift term is not positive, see Corollary 4.4 in \cite{he2016continuous}. 

 We also observe that the linear drift case, i.e $\psi(u)=-bu$ for $u\ge 0$,  when $c>0$ corresponds to the monomorphic model of a single population living in a patchy environment which was  studied recently in Evans et al. \cite{EHS}.

Recently, Leman and Pardo \cite{LP1} studied  the event of extinction  and the  property of  coming  down from infinity of CB-processes  with general competition mechanisms in a L\'evy environment under the assumption that the branching mechanism satisfies the first moment condition \eqref{integralcond1}. 
In particular in \cite{LP1} it is  proved,   under the so-called Grey's condition together  with the assumption that the L\'evy  environment does not drift towards infinity,  that for any starting point the process becomes  extinct in finite time a.s.  Moreover if  the condition on the L\'evy environment is replaced by an integrability condition on the competition mechanism  then the process comes down from infinity.

In this paper,  we study the particular case when the competition mechanism is logistic, where more explicit results about the extinction time
can be provided. In particular, when the branching mechanism is associated to a subordinator, we provide conditions under which $0$ is polar, i.e the process never becomes extinct. Moreover, when  the process does not become extinct, we provide conditions for
the process to be recurrent or transient and give a description of the invariant measure when it exists. 

In order to establish our results, we introduce the following  notation. Let us  denote by $\mathbb{P}_x$, the law of $Z$ starting from $x>0$, and  define the first hitting time to $0$ of $Z$ as follows
\[
T_0=\inf\{t\geq 0, Z_t=0\},
\]
with the convention that $\inf\{\emptyset\}=\infty$.   Hence, $0$ is polar for $Z$ if and only if $\P_x(T_0<\infty)=0$ for all $x>0$. We adopt the following definition of recurrence and transience (see for instance Chapter X of Revuz and Yor \cite{RevuzYor} or Definition 1 in Duhalde et al. \cite{DFM})
\begin{definition}\label{defrectrans}
Assume that $0$ is polar, the process $Z$ is said to be recurrent if there exists $x>0$ such that
\[
\P_x\left(\liminf_{t\to \infty} |Z_t - x| = 0\right) = 1.
\]
On the other hand, the process is said to be transient if
\[
\P_x \left( \lim_{t\to\infty} Z_t=\infty\right)=1, \quad \text{for every} \quad x>0.
\]
\end{definition}
Observe  that if the property of recurrence is satisfied for a particular $x>0$, it is also true for all $x>0$. We also point out that  in Definition 1 of \cite{DFM}, the authors did not assume the polarity of $0$, since they studied a process with positive immigration. In that case, contrary to ours, the process may grow again after extinction and  thus it is either recurrent for all $x\geq 0$, or transient.

For clarity of  exposition, we split our results in two cases  depending on the form of the branching mechanism $\psi$, the subordinator case and what we call the general case which is nothing but the cases where the branching mechanism is associated with a subordinator with negative drift or with an unbounded variation spectrally positive L\'evy process. Both cases use different techniques also. Indeed  in the  subordinator case we use the Lamperti-type representation since the law of the underlying process is known and implicity  many path properties can be established. Unfortunately, this technique cannot be applied in the general case since the law of the underlying process  seems to be not so easy to be determined. Instead, we use a similar approach as in Lambert \cite{Lambert2005} where the knowledge of the infinitesimal generator is relevant.

\subsection{Subordinator case}
Let us assume that the  branching mechanism is associated to the Laplace transform of a subordinator,  that is to say
\begin{equation}\label{subor}
 \psi(z)=-\delta z-\int_{(0,\infty)}(1-e^{-z u})\mu(\ud u),
\end{equation}
where 
\[
\int_{(0,\infty)} (1\land u)\mu (\ud u)<\infty \qquad \textrm{and} \qquad
\delta:=b-\int_{(0,1)} u\mu(\ud u)\ge 0.
\]
We also introduce, the function
\[
\omega(x)=cx+\frac{\sigma^2 x^2}{2},
\]
with $\sigma>0$ and $c\ge 0$. Notice that we are also considering the case without competition (i.e. $c=0$) and implicitly  we will obtain (up to our knowledge) some   unknown path properties for CB-processes in a Brownian environment with  branching mechanism given  by \eqref{subor}. 

Our first  result provides a necessary and sufficient condition under which the process $Z$ is conservative, i.e. that $Z$ does not explode at finite time a.s.

\begin{theorem}\label{teo1} Assume that $\sigma>0$ and $c\ge 0$. The process $Z$, the unique strong solution of \eqref{SDEBL} with branching mechanism given by \eqref{subor}, is conservative if and only if  
\[
\mathcal{I}:=\int_0^1 \frac{1}{\omega (z)}\exp\left\{\int_z^1 \frac{\psi(u)}{\omega(u)}\ud u\right\}\ud z=\infty.
\]
Moreover, if $\sigma^2>2\delta$, then the process $Z$ converges to $0$ with positive probability, i.e
\[
\mathbb{P}_x\left(\lim_{t\to\infty} Z_t=0\right)>0, \qquad \textrm{for}\quad x>0.
\]
In particular,  if we also assume that  $\mathcal{I}=\infty,$ then the process converges to $0$ a.s.
\end{theorem}
For instance, when the branching mechanism is such that  $\psi(z)=-c_\alpha z^\alpha$, for $z\ge 0$, with $\alpha\in(0,1)$ and $c_\alpha>0$, that is to say the negative of a stable subordinator, straightforward computations lead to $\mathcal{I}$ is finite or infinite accordingly as $c=0$ or $c>0$. In other words, if there is presence of competition the associated process $Z$ is conservative and moreover the process becomes extinct a.s., since $\sigma^2$ is always positive. If there is no competition, the process $Z$ explodes with positive probability. The latter case was studied in Palau and Pardo \cite{PP0} where the rate of explosion was determined explicitly.

In this setting, we also have the following identity for the total population size of the process $Z$ up to time $T_a=\inf\{t\ge 0: Z_t\le a\}$, the first hitting time of $Z$ at $a$.  Let us define
\begin{equation*}
f_\lambda(x):= \int_{0}^\infty \frac{\ud z}{\omega(z)}\exp\left\{-xz+\int_\ell^z \frac{\lambda-\psi(u)}{\omega(u)}\ud u\right\}, \qquad x\ge 0,
\end{equation*}
where $\ell$ is an arbitrary constant larger than 0. 
\begin{proposition}\label{prop_duhalde} Assume that $\sigma>0$ and $c\ge 0$.
For every $\lambda>0$ and  $x\ge a\ge 0$, we have
 \begin{equation}
\label{eq_duhalde}
\E_x\left[\exp\left\{-\lambda \int_0^{T_a} Z_s\ud s \right\}\right]=\frac{f_\lambda(x)}{f_\lambda (a)}.
\end{equation}
\end{proposition}

Similarly to the case when the environment is fixed (i.e. $\sigma^2=0$), treated by  Lambert~\cite{Lambert2005},  we observe that  when $c>0$,  the process   $Z$ may have an invariant distribution which can be described explicitly. In order to do so,    we introduce the following notation.
 Let
\begin{equation}\label{defm}
m(\lambda):=\int_0^\lambda \frac{\psi(u)}{\omega(u)}\ud u, \qquad  \textrm{for}\quad\lambda \ge 0, 
\end{equation}
which is  well defined under the log-moment condition~\eqref{logcondition} and  $c>0$ (see  for instance Corollary 3.21 in Li~\cite{li2010}). 

The next Lemma is necessary for the description of the invariant distribution of $Z$, whenever it exists.

\begin{lemma} \label{lemma_mlambda} Assume that $\sigma^2, c>0$ and that the branching mechanism $\psi$, given by \eqref{subor}, satisfies the log-moment condition \eqref{logcondition}.
Then the following identity holds
\begin{equation}\label{eq_lemdecm}
-m(\lambda)=\frac{2}{\sigma^2}\int_0^\infty \Big(1-e^{-\lambda z}\Big)\frac{e^{-\frac{2c}{\sigma^2} z}}{z}\left(\delta +\int_0^z e^{\frac{2c}{\sigma^2} u}\bar{\mu}(u)\ud u \right)\ud z, 
\end{equation}
where $\bar{\mu}(x)=\mu(x,\infty)$, and 
\[
\int_{(0,\infty)} e^{-\lambda z}\nu(\ud z)=e^{m(\lambda)}, \qquad \lambda \ge 0,
\]
defines a unique probability measure $\nu$ on $(0,\infty)$ which is infinitely divisible. In addition, it is self-decomposable whenever $\bar\mu(0)\leq \delta$.
\end{lemma}
We recall that self-decomposable distributions on $(0,\infty)$ is a subclass of infinitely divisible distributions whose L\'evy measures have densities which are decreasing on $(0,\infty)$. We refer to Sato \cite{Sa} for further details on self-decomposable distributions. 

In order to  introduce the limiting distribution associated to $Z$, whenever it exists, we first provide conditions under which $\int_{(0,\infty)} s^{-1}\nu(\ud s)$ is finite. 
For any $z$ sufficiently small, we define two sequences of functions as follows
\begin{equation*}
\begin{aligned}
&l^{(1)}(z)=|\ln(z)|&& \text{and } \quad l^{(k)}(z)=\ln(l^{{(k-1)}}(z)), \qquad k\in \N, k\geq 2,\\
& I^{(1)}(z)=l^{(1)}(z)\int_0^z\bar\mu(w)\ud w&& \text{and } \quad I^{(k)}(z)= l^{(k)}(z)\left(I^{(k-1)}(z)-\frac{\sigma^2}{2}\right), \quad k\in \N, k\geq 2.
\end{aligned}
\end{equation*}
Observe  that for any $k\in \N$, $I^{(k)}(z)$ is well defined for $z$  sufficiently small. On the other hand $l^{(k)}(z)$ is well  defined for  both, $z$  sufficiently small and large. Then, for any continuous function $f$ taking values in $\R$, we set
$$
\texttt{Adh}(f)=\left[\liminf_{z\to 0} f(z), \limsup_{z\to 0} f(z) \right]\subset \R.
$$
We are now ready to establish the following  two conditions which will give the behaviour of $Z$ under the particular setting when $2\delta=\sigma^2$:
\[
(\partial)\;  \text{There exists } n\in \N \text{ such that } \inf({\tt Adh}(I^{(n)}))>\frac{\sigma^2}{2} \text{ and } {\tt Adh}(I^{(k)})=\left\{\frac{\sigma^2}{2}\right\}, \; \textrm{for all } k\in\{1,..,n-1\},
\]
\[
(\eth)\; \text{There exists } n\in \N \text{ such that } \sup({\tt Adh}(I^{(n)}))<\frac{\sigma^2}{2} \text{ and } {\tt Adh}(I^{(k)})=\left\{\frac{\sigma^2}{2}\right\}, \; \textrm{for all } k\in\{1,..,n-1\} .
\]
For instance if $\bar\mu(0)<\infty$ (i.e. $\psi$ is the Laplace exponent of a compound Poisson process)   condition $(\eth)$ holds.
These two conditions are exclusive conditions under which the process is either positive recurrent or null recurrent, that is to say  the process is recurrent and either it has an invariant probability measure or  not.

\begin{theorem}\label{prop_casesubordinator} Assume that $2\delta\ge \sigma^2>0$, $c>0$ . Then the point $0$ is polar, i.e. $
\mathbb{P}_x(T_0<\infty)=0$ for all $x>0$.

Moreover if 
\begin{equation}\label{cond_rec}
 \int_0^1\frac{\ud z}{z} \exp\left\{ -\int_z^1\int_0^{\infty}\frac{(1-e^{-u s})}{\omega(u)}\mu(\ud s)\ud u \right\}=\infty
\end{equation}
$Z$ is recurrent. Additionally, 
\begin{itemize}
\item[a)] if $2\delta>\sigma^2$ then the process $Z$ is positive recurrent. Its invariant distribution $\rho$ has a finite expected value if and only if \eqref{logcondition} holds. If the latter holds, then  $\rho$ is the size-biased distribution of $\nu$, in other words
\begin{equation}\label{def_rho}
\rho(\ud z)=\left(\int_{(0,\infty)} s^{-1}\nu(\ud s)\right)^{-1} z^{-1}\nu(\ud z), \qquad z>0,
\end{equation}
\item[b)] if  $2\delta=\sigma^2$ and \eqref{logcondition} holds, together 
\begin{itemize}
 \item[b.1)]  with  condition $(\partial)$, then $Z$ is positive recurrent and its invariant probability is defined by \eqref{def_rho},
 \item[b.2)]  or  with condition $(\eth)$, then the process $Z$ is null recurrent and converges to $0$ in probability.
\end{itemize}
\end{itemize}
Finally, if \eqref{cond_rec} is not satisfied, then $Z$ explode at finite time a.s. 
\end{theorem}
It is important to note that  \eqref{cond_rec} is satisfied as soon as \eqref{logcondition} holds.  We also point out that the previous results are consistent with  the behaviours found in Proposition 2.1 in Evans et al. \cite{EHS} where $\psi(z)=-bz$.

\subsection{General case}
Finally, we consider the case when the process $X$ is not a subordinator, in other words the branching mechanism  $\psi$ satisfies that there exist $\vartheta \geq 0$ such that $\psi(z)>0$ for any $z\geq \zmin$. For simplicity, we say that the branching mechanism $\psi$ is general if it satisfies the previous assumptions.

In the sequel we assume that $c>0$ and that the L\'evy measure  associated to the  general branching mechanism $\psi$ satisfies  the log-moment condition \eqref{logcondition}. Our main result in this section provides a complete characterization of the  Laplace transform of the stopping times 
\[
T_a=\inf\{t\ge 0: Z_t\le a\}, \quad \text{for} \quad a\geq 0,
\]  
as long as $T_0$ is finite a.s. To this aim, we  introduce  the functional 
\begin{equation}\label{defI}
 {\tt I}(\lambda):=\int_0^\lambda e^{m(u)}\ud u, \qquad \textrm{for}\quad \lambda\ge 0,
\end{equation}
where $m$ is defined by \eqref{defm} and well posed under the log-moment \eqref{logcondition}. Observe from our assumptions that $m$ is increasing on $(\zmin,\infty)$ implying that ${\tt I}(\cdot)$ is a bijection from $\R_+$ into itself. We denote its inverse by $\varphi$ and a simple computation provides
 \begin{equation}
 \label{eq_derphi}
 \varphi'(z)=\exp(-m\circ \varphi(z)).
\end{equation}
The formulation of the Laplace transform of $T_a$ will be written in terms of the solution to a Ricatti equation.
Similarly to  Lemma 2.1 in \cite{Lambert2005}, we deduce the following Lemma on the Ricatti equation of our  interest.
 \begin{lemma}
 \label{lemma_hq}
 For any $\lambda>0$, there exists a unique non-negative solution $y_\lambda$ to the equation
 \begin{equation}
 \label{eq_h}
 y'=y^2-\lambda r^2,
 \end{equation}
where $r(z)=\frac{ \varphi'(z)}{\sqrt{\omega(\varphi(z))}}$ such that it vanishes at $\infty$. Moreover, $y_\lambda$ is positive on $(0,\infty)$, and for any $z$ sufficiently small or
large, $y_\lambda(z) \leq \sqrt{\lambda}r(z)$. As a consequence, $y_\lambda$ is integrable at $0$, and it decreases initially and ultimately.
 \end{lemma}

We now state our last result. Recall that the infinitesimal generator $\mathcal{U}$  of the process $Z$ satisfies  that for any $f\in C^2$, 
\begin{equation}\label{infgenerator}
\mathcal{U} f(x)=(bx-cx^2)f^\prime(x)+\left(\gamma^2 x+\frac{\sigma^2}{2}x^2\right)f^{\prime\prime}(x)+x\int_{(0,\infty)}\left(f(x+z)-f(x)-zf^\prime(x)\mathbf{1}_{\{z<1\}}\right)\mu(\ud z),
\end{equation}
see for instance Theorem 1 in Palau and Pardo \cite{PP}.

\begin{theorem}
\label{theo_Tatteinte} Let $c>0$ and assume  that the branching mechanism $\psi$ is general and its associated L\'evy measure  satisfies the log-moment condition \eqref{logcondition}. Hence the function
\begin{equation}\label{eq-theo}
h_{\lambda}(x):=1+\lambda \int_0^{\infty} \frac{e^{-xz-m(z)}}{\omega(z)} \exp\left\{ -\int_0^{{\tt I}(z)}y_\lambda(v)\ud v \right\} \int_0^z \exp\left\{ m(u)+2\int_0^{{\tt I}(u)}y_\lambda(v)\ud v \right\}\ud u \ud z
\end{equation}
is well defined and positive for any $x>0$ and $\lambda > 0$ and it is a non-increasing $C^2$-function on $(0,\infty)$. Moreover it solves
\begin{equation}
\label{eq_vpgenerator}
\mathcal{U} h_\lambda(x)=\lambda h_\lambda(x), \text{ for any } x>0.
\end{equation}
Furthermore, if 
$\mathbb{P}_x(T_0<\infty)=1$, for any $x>0$
then  $h_\lambda$ is also well-defined at $0$ with
$$h_\lambda(0)=\exp\left\{\int_0^{\infty} y_\lambda(v)dv\right\}<\infty,$$
and, for any $x\geq a\geq 0$,
\begin{equation}
\label{eq_resultLaplace}
\E_x\left[ e^{-\lambda T_a}\right]=\frac{h_\lambda(x)}{h_\lambda(a)}.
\end{equation}
In particular,  for any $x>0$,
\begin{equation}
 \label{eq_resultexp}
 \E_x[T_0]=\int_0^{\infty}\ud u \,e^{m(u)}\int_u^{\infty}\frac{e^{-m(z)}}{\omega(z)}(1-e^{-zx})\ud z.
\end{equation}
\end{theorem}

It is important to note that under the assumptions of Theorem 1.2 in \cite{LP1}, the previous result  can be applied. To be more precise,   according to Theorem 1.2 in \cite{LP1}  if $\psi$ satisfies Grey's condition \eqref{grey}  together with \eqref{integralcond1} (i.e. $|\psi'(0+)|<\infty$), then $\E_x[T_0]<\infty$, for any starting point $x\ge 0$. In other words, under these assumptions, the results of Theorem~\ref{theo_Tatteinte} apply and moreover the logistic branching process in a Brownian environment $Z$ is Feller and comes down from infinity as it is stated in the following Corollary.
Formally,  we define the property of {\it coming down from infinity} in the sense that $\infty$ is a {\it continuous entrance point}, i.e.
\[
\lim_{a\to \infty}\lim_{x\to \infty}\mathbb{P}_x(T_a<t)=1 \qquad \textrm{for all} \quad t>0, 
\]
 and 
the original process can be extended into a Feller process on  $[0,\infty]$  (see for instance Theorem 20.13 in Kallenberg~\cite{Kallenberg} for the diffusion case  or Definition 2.2 for Feller processes in \cite{DK}). 
\begin{corollary}\label{corollary1} Assume that $c>0$ and that the branching mechanism $\psi$ is general and satisfies \eqref{integralcond1} and Grey's condition \eqref{grey}. Then the logistic branching process in a Brownian environment $Z$ is Feller and  the  boundary point $\infty$ is a continuous entrance point. Moreover, the process $Z$ can be extended into a Feller process on $[0,\infty]$ and, in particular,    we have
\[
\E_\infty\left[ e^{-\lambda T_a}\right]=\frac{1}{h_\lambda(a)}\qquad \textrm{ and }\qquad\E_\infty[T_0]=\int_0^{\infty}\ud u \,e^{m(u)}\int_u^{\infty}\frac{e^{-m(z)}}{\omega(z)}\ud z.
\]
\end{corollary}
We believe that $T_0$ is finite  a.s., under much weaker conditions (including the case $-\psi'(0+)=\infty$) than those stated  in Corollary \ref{corollary1} but in order to deduce such result the knowledge of the underlying process in the Lamperti-type representation is necessary. Under such weaker conditions we can also expect that the process $Z$ must be Feller which can be extended to $[0,\infty]$.

The remainder of this  paper is organised as follows. In Section 2, we deal with  a Lamperti-type representation  which is established for more general competition mechanisms than the logistic case. Such random time change representation  is very useful for the proofs of the subordinator case which are presented in Section 3. Section 4 is devoted to the proof of the results presented for the general case which uses the solution of Ricatti differential equation that appears in Lemma \ref{lemma_hq}. Finally, in Section 5 we discuss the case when the competition mechanism is more general and the process possesses  continuous paths. We call this case  branching diffusions with interactions in a Brownian random environment,  since  the competition mechanism  may take negative and positive values. We study this case separately  since the techniques we use here  are based on the theory of scale functions for diffusions. This allow us to provide a necessary and  sufficient condition  for extinction and moreover, the Laplace transform of  hitting times is computed explicitly in terms of a Ricatti equation. Such results seems complicated to obtain with the presence of jumps coming from the branching mechanism.

 \section{Lamperti-type transform for CB-processes with competition in a Brownian environment.}\label{Lamperti}
 
Let  $g$ be a continuous function on $[0,\infty)$ with $g(0)=0$ and consider the following SDE 
\begin{equation}\label{SDEB}
\begin{split}
Z_t&=Z_0+b\int_0^tZ_s \ud s-\int_0^tg(Z_s)\ud s+\int_0^t\sqrt{2\gamma^2Z_s}\ud B^{(b)}_s+\sigma\int_0^t Z_{s}\ud  B^{(e)}_s\\
&\hspace{3cm}+\int_0^t\int_{[1,\infty)} \int_0^{Z_{s-}}z{N}^{(b)}(\ud s, \ud z, \ud u) +\int_0^t\int_{(0,1)} \int_0^{Z_{s-}}z\widetilde{N}^{(b)}(\ud s, \ud z, \ud u),
\end{split}
\end{equation}
with $\sigma \ge 0$. It is important to note that Proposition 1 in Palau and Pardo \cite{PP} guarantees that the above SDE has a unique strong positive solution up to explosion and by convention here it is identically equal to $+\infty$ after the explosion time.

The main result in this section is the Lamperti-type representation of  a CB-process with competition in a Brownian environment. Such random time change representation will be very useful to study path properties of the logistic case. In order to state the Lamperti-type representation, we introduce the family of processes which are involved in the time change.   

Let $X=(X_t,t\ge 0)$ be a  spectrally positive L\'evy process with characteristics $(-b, \gamma, \mu)$ and such that its   L\'evy measure $\mu$ satisfies (\ref{mugral}). We also consider $W=(W_t, t\ge 0)$ a standard Brownian motion independent of $X$  and assume that  $g$ is a continuous function on $[0,\infty)$ with $g(0)=0$ and such that  $\lim_{x\to 0}x^{-1}g(x)$ exists. According to Proposition 1 in Palau and Pardo \cite{PP} for each $x>0$, there is a unique strong solution to
\begin{equation}\label{SDE1}
\ud R_t=\mathbf{1}_{\{R_{r-}>0:r\le t\}}\ud X_t-\mathbf{1}_{\{R_{r-}>0:r\le t\}}\frac{g(R_t)}{R_t}\ud  t+\mathbf{1}_{\{R_{r-}>0:r\le t\}}\sigma\sqrt{ R_t}\ud W_t ,
\end{equation}
with $R_0=x$. The assumption that  $\lim_{x\to 0}x^{-1}g(x)$  exists, is not necessary but it implies that we can use directly Proposition 1 of Palau and Pardo \cite{PP}.  We can relax this assumption but further explanations are needed. Indeed  a similar approach to Theorems 2.1 and 2.3 in Ma \cite{Ma2015} will guarantee that the SDE defined above for a more general competition mechanism $g$ has a unique strong solution.

It is important to note that in  the logistic-case i.e. $g(x)=cx^2$, for $x\ge0$ and some constant $c>0$, the process $R$ is  a Feller diffusion which is perturbed by the L\'evy process $X$. Moreover if the L\'evy process $X$ is a subordinator, then the process $R$ turns out to be  a CB-process with immigration. 

We now state the Lamperti-type representation of CB-processes with competition in a Brownian environment. 
\begin{theorem}\label{lamperti}
 Let $R=(R_t,t\ge 0)$ be the unique strong solution of (\ref{SDE1}) and  $T^R_0=\sup\{s: R_s=0\}$. We also let $C$ be  the right-continuous inverse of $\eta$, where
\[
\eta_t=\int_0^{t\land T^R_0}\frac{\ud s}{R_s}, \qquad t>0,
\]
that is, $C_t:=\inf\{s\geq 0, \eta_s > t \}$, for any $t\in [0,+\infty)$.
Hence  the process defined by 
\[
Z_t=\left\{ \begin{array}{ll}
R_{C_t}, & \textrm{ if } 0\le t<\eta_\infty\\
0, & \textrm{ if } \eta_\infty<\infty, T^R_0<\infty \textrm{ and } t\ge \eta_\infty,\\
+\infty, & \textrm{ if } \eta_\infty<\infty, T^R_0=\infty \textrm{ and } t\ge \eta_\infty,
\end{array}
\right .
\]
satisfies the SDE (\ref{SDEB}).

Reciprocally, let $Z$ be the unique strong solution to (\ref{SDEB}) with $Z_0=x$ and let 
\[
C_t=\int_0^{t} Z_s\ud s, \qquad t>0.
\]
If $\eta$ denotes the right-continuous inverse of $C$, then the process defined by 
\[
R_t=Z_{\eta_t\land T_0} \qquad\textrm{for}\quad t\ge0.
\]
satisfies the SDE (\ref{SDE1}).
\end{theorem}

\begin{proof}[Proof of Theorem~\ref{lamperti}] Since $X$ is a spectrally positive L\'evy process  and $R_{t-}=0$ implies $R_t=0$, we get $R_{t-}>0$ if and only if $t\in [0, T_0^R).$  We also observe that $X$ can be written as follows
\[
X_t=bt+\sqrt{2}\gamma B_t+\int_0^t\int_{(0,1)} z\widetilde{M}(\ud s, \ud z) +\int_0^t\int_{[1,\infty)} z M (\ud s, \ud z),
\]
where $B$ is a standard Brownian motion and  $M$ is a Poisson random measure with intensity $\ud s \mu(\ud z)$ and $\widetilde{M}$ denotes its compensated version. Then from the latter identity and  (\ref{SDE1}), we have
\[
\begin{split}
Z_t&=x+b\int_0^{C_t\land T_0^R} \ud s-\int_0^{C_t\land T_0^R} \frac{g(R_s)}{R_s}\ud s+\sqrt{2}\gamma\int_0^{C_t\land T_0^R} \ud {B}_s +\int_0^{C_t\land T_0^R}\sigma\sqrt{ R_s}\ud W_s\\
&\hspace{2cm}+\int_0^{C_t\land T_0^R}\int_{(0,1)} z\mathbf{1}_{\{R_{s-}>0\}}\widetilde{M}(\ud s, \ud z) +\int_0^{C_t\land T_0^R}\int_{[1,\infty)} z\mathbf{1}_{\{R_{s-}>0\}}M(\ud s, \ud z), \quad t\ge 0. 
\end{split}
\]
On the one hand, by straightforward computations we deduce
\[
C_t\land T_0^R=\int_0^t Z_s \ud s, 
\]
implying that 
\[
\int_0^{C_t\land T_0^R} \frac{g(R_s)}{R_s}\ud s=\int_0^t g(Z_s)\ud s, 
\]
and 
\[
L^{(1)}_t=\sqrt{2}\gamma\int_0^{C_t\land T_0^R} \ud {B}_s  \qquad \textrm{and} \qquad L_t^{(2)}=\sigma\int_0^{C_t\land T_0^R}\sqrt{ R_s}\ud W_s,
\]
are independent continuous local martingales with  increasing processes
\[
\langle L^{(1)}\rangle_{t}=2\gamma^2\int_0^t Z_s\ud s\qquad \textrm{and}\qquad \langle L^{(2)}\rangle_{t}=\sigma^2\int_0^t Z^2_s\ud s.
\]
On the other hand, we define the random measure $N(\ud s, \ud z)$ on $(0,\infty)^2$ as follows
\[
N((0,t]\times \Lambda)=\int_0^{C_t\land T_0^R}\int_{(0,\infty)} \mathbf{1}_\Lambda(z)\mathbf{1}_{\{R_{s-}>0\}}M(\ud s, \ud z).
\]
Then $N(\ud s, \ud z)$ has predictable compensator
\[
Z_{s-}\ud s\mu(\ud z).
\]
By Theorems 7.1 and 7.4 in Ikeda and Watanabe \cite{IW89}, on an extension of the original probability space there exist two independent Brownian motions, $B^{(1)}$ and $B^{(2)}$, and a Poisson random measure $N(\ud s, \ud u, \ud z)$ on $(0,\infty)^3$ with intensity $\ud s\mu(\ud z)\ud u$ such that for any $t\ge 0$,
\[
\int_0^{C_t\land T_0^R}\int_{[1,\infty)} z\mathbf{1}_{\{R_{s-}>0\}}M(\ud s, \ud z)=\int_0^t \int_{[1,\infty)} \int_0^{Z_{s-}}zN(\ud s, \ud z, \ud u),
\]
\[
\int_0^{C_t\land T_0^R}\int_{(0,1)} z\mathbf{1}_{\{R_{s-}>0\}}\widetilde{M}(\ud s, \ud z)=\int_0^t\int_{(0,1)}\int_0^{Z_{s-}}  z\widetilde{N}(\ud s,\ud z, \ud u),
\]
\[
L^{(1)}_t=\int_0^t \sqrt{2\gamma^2 Z_s}\ud B^{(1)}_s\qquad\textrm{and} \qquad L^{(2)}_t=\sigma \int_0^t Z_s \ud B^{(2)}_s.
\]
Putting all pieces together, we deduce that $(Z_t, t\ge 0)$ is a solution of (\ref{SDEB}) up to explosion.

For the reciprocal, we first observe that since $Z$ has no negative jumps and $Z_{t-}=0$ implies $Z_t=0$, we get $Z_{t-}>0$ if and only if $Z_t>0$ for $t\in [0, T_0)$. Thus $R_{t-}>0$ if and only if $R_t>0$ for $t\in [0,C_{T_0})$, then for any $t\in [0, C_{T_0})$, the equation (\ref{SDE1}) is equivalent to 
\begin{equation}\label{SDE2}
R_t=\ud X_t-\frac{g(R_t)}{R_t}\ud  t+\sigma\sqrt{ R_t}\ud W_t. 
\end{equation}
Since the process $Z$ satisfies the SDE (\ref{SDEB}) and $R_t=Z_{\eta_t \land T_0}$, we have
\begin{equation}\label{SDE3}
\begin{split}
R_t=&Z_0+b\int_0^{\eta_t\land T_0} Z_s\ud s +\int_0^{\eta_t\land T_0}\sqrt{2\gamma^2Z_s}\ud B_s+\sigma\int_0^{\eta_t\land T_0}Z_{s}\ud B^{(e)}_s\\
&+\int_0^{\eta_t\land T_0}\int_{[1,\infty)}\int_0^{Z_{s-}} z{N}(\ud s, \ud z, \ud u) +\int_0^{\eta_t\land T_0}\int_{(0,1)}\int_0^{Z_{s-}} z\widetilde{N}(\ud s, \ud z, \ud u)-\int_0^{\eta_t\land T_0}g(Z_s)\ud s.
\end{split}
\end{equation}
On the one hand, by straightforward computations we deduce
\[
\int_0^{\eta_t\land T_0} Z_s \ud s=t\land C_{T_0}, \qquad \textrm{and} \qquad \int_0^{\eta_t\land T_0} g(Z_s) \ud s=\int_0^{t\land C_{T_0}} \frac{g(R_s)}{R_s} \ud s.
\]
The latter identities imply 
\[
M^{(1)}_t=\int_0^{\eta_t\land T_0}\sqrt{2\gamma^2Z_s}\ud B_s  \qquad \textrm{and} \qquad M_t^{(2)}=\sigma\int_0^{\eta_t\land T_0}Z_{s}\ud B^{(e)}_s,
\]
are independent continuous local martingales with  increasing processes
\[
\langle M^{(1)}\rangle_{t}=2\gamma^2\int_0^{\eta_t\land T_0}Z_s\ud s=2\gamma^2 (t\land C_{T_0}) \qquad \textrm{and}\qquad \langle M^{(2)}\rangle_{t}=\sigma^2\int_0^{\eta_t\land T_0} Z^2_s\ud s=\sigma^2\int_0^{t\land C_{T_0}} R_{s}\ud s.
\]
By Theorems 7.1 and 7.4 in Ikeda and Watanabe \cite{IW89}, on an extension of the original probability space there exist two independent Brownian motions, $B^{(1)}$ and $B^{(2)}$, such that for any $t\ge 0$,
\begin{equation}\label{MBs}
M^{(1)}_t= B^{(1)}_{t\land C_{T_0}}\qquad\textrm{and} \qquad M^{(2)}_t=\sigma \int_0^{t\land C_{T_0}} \sqrt{R_{s}}\ud B^{(2)}_s.
\end{equation}
On the other hand, we define the random measure $M(\ud s, \ud z)$ on $(0,\infty)^2$ as follows
 \begin{equation}\label{MP}
M((0,t]\times \Lambda)=\int_0^{\eta_t\land T_0}\int_{(0,\infty)}\int_0^{Z_{s-}} \mathbf{1}_{\Lambda}(z){N}(\ud s, \ud z, \ud u)+ \int_{C_{ T_0}}^{t}\int_{(0,\infty)}\int_{0}^{1} \mathbf{1}_{\Lambda}(z)\mathbf{1}_{\{t>C_{T_0}\}}{N}(\ud s, \ud z, \ud u).
\end{equation}
Then $M(\ud s, \ud z)$ has predictable compensator
$\ud s\mu(\ud z).$ Hence, $M(\ud s, \ud z)$ is a Poisson random measure on $(0,\infty)^2$ with intensity $\ud s\mu(\ud z)$.
Putting all the pieces together, we deduce that  (\ref{SDE2}) holds for $t\in[0,C_{T_0})$. Recall that $Z_{T_{0}-}=Z_{T_{0}}=0$. Then on $\{C_{T_0}<\infty\}$ by using  (\ref{SDE3})-(\ref{MBs}), we deduce that the right hand side of (\ref{SDE2}) is equal to 0 for $t=C_{T_0}$ and then for all $t\ge C_{T_0}$. 
\end{proof} 

\section{Proofs of the subordinator case}

In this part, we provide the proofs of  Theorems \ref{teo1} and ~\ref{prop_casesubordinator}.  Their proof relies on  the Lamperti-type representation in the  discussed in the previous section. Unfortunatelly, the same techniques cannot be used in the general case since a deep understanding of the process $R$ is required such as its marginal laws and path properties as recurrence and transience which seems  not so clear  to deduce.

In the particular case when the spectrally positive L\'evy process $X$ is a subordinator in the Lamperti-type representation in Theorem \ref{lamperti}, the process $R$   turns out to be  a Feller diffusion with immigration. In other words, it is the unique positive strong solution of the following SDE up to the first hitting time of $0$:
\begin{equation}\label{CBI}
R_t= R_0+X_t-c \int_0^tR_s\ud s+\int_0^t\sqrt{\sigma^2 R_s}\ud W_s.
\end{equation}
The branching mechanism $\omega$ and the immigration mechanism $\phi$ associated to the process $R$, are given by
\[
\omega(z)=cz+\frac{\sigma^2 z^2}{2} \quad\textrm{ and }\quad \phi(z)=-\psi(z)=\delta z+\int_{(0,\infty)}(1-e^{-z u})\mu(\ud u),
\]
respectively and where 
\[
\int_{(0,\infty)} (1\land u)\mu (\ud u)<\infty \qquad \textrm{and} \qquad
\delta=b-\int_{(0,1)} u\mu(\ud u)\ge 0.
\]
We denote by $\mathbb{Q}_x$, for the law of the Feller diffusion with immigration  $R$ starting from $x>0$.

This type of processes have been studied recently by many authors, see for instance the papers of  Keller-Ressel and Mijatovic \cite{KRM} and Duhalde et al. \cite{DFM} and the references therein. In \cite{KRM}, the authors were interested in the invariant  distribution associated to the process $R$ and Duhalde et al. \cite{DFM} studied  first passage times problems   and provide necessary and sufficient conditions for polarity and recurrence.  

\begin{lemma}\label{polarity} Let  $R=(R_t; t\ge 0)$ be the Feller diffusion with immigration described by  \eqref{CBI} with branching and immigration mechanisms given by $\omega$ and $\phi$, respectively. The point 0 is polar, i.e. $T^R_0=\infty$ almost surely,  if and only if  $2\delta\ge \sigma^2$.  
\end{lemma}
\begin{proof}
 According to Theorem 2 in Duhalde et al.~\cite{DFM}, the point $0$ is polar for the Feller diffusion with immigration $R$, accordingly as 
\begin{equation}\label{cond_polarduhalde}
 \int_{1}^\infty \frac{\ud \lambda}{\omega(\lambda)}\exp\left\{\int_1^\lambda \frac{\phi(z)}{\omega(z)}\ud z\right\}=\infty.
\end{equation}
Let $K:=2c/\sigma^2$,  which is equal to 0 if $c=0$. Then for any $\lambda > 1$ and $x_0>0$, we have
\begin{equation}\label{eq_m1}
\int_{1}^{\lambda}\frac{\phi(z)}{\omega(z)}\ud z=\frac{2}{\sigma^2} \int_1^\lambda \left( \frac{\delta z}{K z+z^2} +\frac{1}{K z+z^2}\int_0^{\infty}(1-e^{-zu})\mu(\ud u)\right)\ud z.
\end{equation}
Since all terms in~\eqref{eq_m1} are positive, we can separate the above integral into two terms and study each of them independently. Then
\[
\begin{split}
\int_{1}^{\lambda}\frac{\phi(z)}{\omega(z)}\ud z&= \frac{2\delta}{\sigma^2}\ln\left(\frac{K+\lambda}{K+1}\right)+\frac{2}{\sigma^2}\int_0^{\infty}\mu(\ud u)\int^{\lambda}_1\frac{1-e^{-zu}}{Kz+z^2}\ud z\\
&\leq   \frac{2\delta}{\sigma^2}\ln\left(\frac{K+\lambda}{K+1}\right)+\frac{2}{\sigma^2}\int_0^{x_0}\mu(\ud u)\int^{\lambda}_1\frac{zu}{Kz+z^2}\ud z +\frac{2}{\sigma^2}\int_{x_0}^{\infty}\mu(\ud u)\int^{\lambda}_1\frac{1}{z^2}\ud z\\
&\leq   \frac{2\delta}{\sigma^2}\ln\left(\frac{K+\lambda}{K+1}\right)+\frac{2}{\sigma^2}\left(\int_0^{x_0}u\mu(\ud u)\right)\ln\left(\frac{K+\lambda}{K+1}\right) +\frac{2}{\sigma^2}\bar\mu(x_0),
\end{split}
\] 
where we used Fubini-Tonelli's theorem to obtain the first equality.
The above inequality holds for any $x_0>0$, hence for any $\varepsilon>0$, we can choose $x_0>0$ such that 
\[
\int_0^{x_0}u\mu(\ud u)\leq \frac{\sigma^2}{2}\varepsilon.
\] 
Then for any $\lambda> 1$, the following inequalities hold
\[
K_1(x_0)\frac{\left({K+\lambda}\right)^{\frac{2\delta}{\sigma^2}}}{\lambda^2}\leq \frac{1}{\omega(\lambda)}\exp\left\{\int_1^\lambda \frac{\phi(z)}{\omega(z)}\ud z\right\} \leq K_2(x_0)\frac{\left({K+\lambda}\right)^{\frac{2\delta}{\sigma^2}+\varepsilon}}{\lambda^2},
\]
where $K_1(x_0)$ and $K_2(x_0)$ are positive constants which are independent from $\lambda$. Therefore we conclude that \eqref{cond_polarduhalde} holds if and only if $2\delta \ge \sigma^2$.
\end{proof}

\begin{proof}[Proof of Theorem \ref{teo1}] We first treat the case $\sigma^2>2\delta$. From Lemma \ref{polarity}, we observe that $0$ is not polar, meaning that the Feller diffusion with immigration $R$ hits $0$ with positive probability. From Theorem \ref{lamperti}, we then deduce
\[
\mathbb{P}_x\left(\lim_{t\to \infty} Z_t=0\right)\geq \mathbb{Q}_x(T_0^R<\infty)>0, \qquad x>0.
\]
In other words, with positive probability, the process $Z$ does not explode. Moreover,  if $\mathcal{I}=\infty$, Theorem 3 in Duhalde et al. \cite{DFM} implies that the process $R$, the unique strong solution to \eqref{CBI},   is recurrent in the sense of Duhalde et al. \cite{DFM} (i.e. without assuming the polarity of $0$, cf. remark after Definition~\ref{defrectrans}). In other words, since $0$ is not polar, $R$ hits $0$ at finite time a.s. Since we are interested in the  unique strong solution of \eqref{CBI} up to the first hitting time of $0$, the latter probability equals 1, i.e. the process $Z$ converges to $0$ a.s.

Next, we assume $2\delta \ge \sigma^2$. From  Lemma \ref{polarity}, we know that   $T^R_0=\infty$ a.s. and thus $\eta_t= \int_0^t \frac{1}{R_s}\ud s$ for any $t \geq 0$.
If we also assume that $\mathcal{I}=\infty$, then the solution to \eqref{CBI} is recurrent and $0$ is polar. Let us thus prove that the limit $\eta_{\infty}$ of $(\eta_t,t\geq 0)$ is $\infty$ a.s. If we define recursively the sequences of finite stopping times as follows $\tau^+_0=0$, and for any $k\ge 1$,
\[ \tau^-_{k+1}=\inf\{t\geq\tau^+_k, R_s\leq 1 \}\qquad  \text{ and } \qquad \tau^+_{k+1}=\inf\{t\geq\tau^-_{k+1}, R_s\geq 2 \},
 \]
we deduce that, since $\{{\tau^+_k}-{\tau^-_k}, k\ge 1\}$ is an infinite sequence of strictly positive i.i.d random variables,
\begin{equation}\label{eq_limeta}
\eta_{\infty}= \int_0^{\infty} \frac{1}{R_s}\ud s \geq \sum_{k\geq 1}\frac{1}{2} ({\tau^+_k}-{\tau^-_k})=\infty, \qquad \textrm{a.s.}
\end{equation}
 This implies  that $C_t$,  the right inverse of $\eta_t$, is well defined on $(0,\infty)$ and that $Z_t=R_{C_t}$ for any $t\geq 0$. In other words, the process $Z$ is conservative.
 
 If $\mathcal{I}<\infty$, then the process $R$ is transient  according to Theorem 3 in  Duhalde et al. \cite{DFM}. Recall that the Laplace transform of $R_t$ satisfies
 \[
 \mathbb{Q}_x[e^{-\lambda R_t}]=\exp\left\{-x v_{t}(\lambda)-\int_0^t \phi(v_s(\lambda))\ud s\right\}, \qquad \textrm{for} \quad \lambda \ge 0,
 \]
 where $v_t(\lambda)$ is  solution of 
 \begin{equation}\label{vdiffeq}
\frac{\partial}{\partial t} v_t(\lambda)= -\omega(v_t(\lambda)), \qquad \textrm{with}\qquad v_0(\lambda)=\lambda. 
 \end{equation}
From the form of the branching mechanims $\omega$ and the previous identity, we deduce
 \[
 v_t(\lambda)=\frac{\lambda e^{-ct}}{1+\frac{\sigma^2 \lambda }{2c}(1-e^{-ct})}, \qquad \textrm{for}\quad t, \lambda\ge 0.
 \]
  Therefore, by Tonelli's Theorem, identity \eqref{vdiffeq}, the fact that $v_\infty(\lambda)=0$ and using twice the change of variables $y=v_t(\lambda)$, we deduce that for $\theta>0$
 \[
 \begin{split}
 \mathbb{Q}_x\left[\int_0^\infty \frac{1-e^{-\theta R_s}}{R_s}\ud s\right]&= \int_0^\theta \ud \lambda \int_0^\infty \ud s\,\mathbb{Q}_x\left[e^{-\lambda R_s} \right]\\
 &=\int_0^\theta \ud \lambda \int_0^\lambda\frac{\ud u}{\omega (u)} \exp\left\{-x u-\int^\lambda_{u}\frac{\phi(y)}{\omega(y)}\ud y\right\} ,
 \end{split}
 \]
 which is clearly finite from our hypothesis. Since the Feller diffusion with immigration $R$  is transient, it is clear that
 \[
 \lim_{s\to \infty} e^{-\theta R_s}=0, \qquad \mathbb{Q}_x\textrm{-a.s.,}
 \]
implying that 
\[
 \mathbb{Q}_x\left[\int_0^\infty \frac{1}{R_s}\ud s\right]<\infty, 
\]
and implicitly the process $Z$ explodes at finite time a.s. This completes the proof.
\end{proof}

We now proceed with the proofs of Proposition \ref{prop_duhalde}, Lemma~\ref{lemma_mlambda} and Theorem \ref{prop_casesubordinator} where it is assumed that $c>0$.

\begin{proof}[Proof of Proposition \ref{prop_duhalde}]
The proof of this result is a direct consequence of the  Lamperti-type representation (Theorem \ref{lamperti}) and Theorem 1 in Duhalde et al. \cite{DFM}.
\end{proof}

\begin{proof}[Proof of Lemma~\ref{lemma_mlambda}]
We first recall that $m$,  introduced in \eqref{defm},  is well defined under the log-moment condition \eqref{logcondition}. Then, similarly to \eqref{eq_m1}, we have 
\begin{equation}\label{eq_m}
-m(\lambda)=\int_0^\lambda \frac{\phi(z)}{\omega(z)}\ud z =\frac{2}{\sigma^2} \int_0^\lambda  \frac{\delta z }{\frac{2c}{\sigma^2}z+z^2} \ud z +\int_0^\lambda\left(\frac{1}{\frac{2c}{\sigma^2}z+z^2}\int_0^{\infty}(1-e^{-zu})\mu(\ud u)\right)\ud z.
 \end{equation}
For simplicity in exposition,  we study the two last integrals independently. 
For the first integral of \eqref{eq_m}, we observe
 \begin{equation*}
  \begin{aligned}
   \int_0^\lambda  \frac{\delta z}{Kz+z^2} \ud z 
   = \delta  \int_0^{\lambda}  \int_0^{\infty}e^{-v(z+K)}\ud v \ud z
   = \int_0^{\infty} (1-e^{-\lambda v}) \frac{\delta e^{-Kv}}{v} \ud v,
  \end{aligned}
 \end{equation*}
where  $K:=2c/\sigma^2$ and the last equality follows from an application of Fubini-Tonelli's theorem. For the second integral of \eqref{eq_m}, we use again  Fubini-Tonelli's theorem, to deduce
\[
 \int_0^\lambda\frac{1}{Kz+z^2}\left(\int_0^{\infty}(1-e^{-zu})\mu(\ud u)\right) \ud z =\frac{1}{K}  \int_0^{\infty}  \left(\int_0^\lambda \frac{K(1-e^{-zu})}{Kz+z^2}\ud z\right)\mu(\ud u).
\]
Now, we fix $u>0$ and study the integral inside the brackets. Since the map $z\mapsto (1-e^{-zu})/z$ is integrable at $0$, we have
\begin{equation*}
 \begin{aligned}
  \int_0^\lambda \frac{K(1-e^{-zu})}{Kz+z^2}&\ud z= \int_0^\lambda \left(\frac{1-e^{-zu}}{z}-\frac{1-e^{-zu}}{K+z}\right)\ud z \\
  &= \int_0^u \frac{1-e^{-\lambda v}}{v}\ud v -\int_K^{K+\lambda}\frac{1-e^{Ku}e^{-zu}}{z}\ud z\\
  &=\int_0^u \frac{1-e^{-\lambda v}}{v}\ud v -(1-e^{Ku})\int_K^{K+\lambda}\frac{1}{z}\ud z-e^{Ku}\int_K^{K+\lambda}\frac{1-e^{-zu}}{z}\ud z\\
  &= \int_0^u \frac{1-e^{-\lambda v}}{v}\ud v-(1-e^{Ku})\int_K^{K+\lambda}\int_0^{\infty}e^{-zv}\ud v \ud z \\
  &\hspace{4cm}+e^{Ku}\left(\int_0^{K}\frac{1-e^{-zu}}{z}\ud z-\int_0^{K+\lambda}\frac{1-e^{-zu}}{z}\ud z\right)\\
  &=\int_0^u \frac{1-e^{-\lambda v}}{v}\ud v -(1-e^{Ku})\int_0^{\infty}\frac{e^{-Kv}}{v}(1-e^{-\lambda v})\ud v  -e^{Ku}\int_0^{u}\frac{e^{-Kv}}{v}(1-e^{-\lambda v})\ud v\\
  &=\int_0^u \frac{1-e^{-\lambda v}}{v}(1-e^{-K v})\ud v +(e^{Ku}-1) \int_u^{\infty}\frac{1-e^{-\lambda v}}{v} e^{-Kv}\ud v
 \end{aligned}
\end{equation*}
where the second identity follows from the  change of variables $zu=\lambda v$, the third identity is obtained by adding and subtracting  $e^{Ku}$,  the fifth identity follows from Fubini-Tonelli's  Theorem and the  change of variables $Kv=zu$ and $(K+\lambda)v=zu$ and finally, the last identity follows by adding and subtracting
\[
\int_0^u \frac{1-e^{-\lambda v}}{v}e^{-K v}\ud v.
\]
In other words, we get
\[
  \int_0^\lambda \frac{K(1-e^{-zu})}{Kz+z^2}\ud z  =\int_0^{\infty} \frac{1-e^{-\lambda v}}{v}e^{-Kv}(e^{K(v\wedge u)}-1)\ud v=\int_0^{\infty} \frac{1-e^{-\lambda v}}{v}e^{-Kv}\left(\int_0^{v \wedge u}Ke^{Kz}\ud z\right)\ud v.
  \]
Putting all pieces together and using twice Fubini-Tonelli's theorem, we obtain the following expression for the second integral of \eqref{eq_m}
\begin{align}
  \int_0^\lambda \frac{1}{Kz+z^2}\left(\int_0^{\infty}(1-e^{-zu})\mu(\ud u)\right)\ud z &= \int_0^{\infty} \frac{1-e^{-\lambda v}}{v} e^{-Kv}\left( \int_0^{\infty}\left(\int_0^{v \wedge u}e^{Kz}\ud z\right)\mu(\ud u)\right)\ud v \nonumber \\
  &=  \int_0^{\infty} \frac{1-e^{-\lambda v}}{v} e^{-Kv}\left( \int_0^{v }e^{Kz}\bar{\mu}(z)\ud z\right)\ud v.\label{eq_m2ndterm}
  \end{align}
Finally from identity~\eqref{eq_m} and the previous computations, we find~\eqref{eq_lemdecm}.

Next, we define the positive measure $\Pi(\ud z)$ as follows
\begin{equation}\label{def_Pih}
\Pi(\ud z)=\frac{2 e^{-K z}}{\sigma^2 z}\left(\delta+\int_0^z e^{K v}\bar{\mu}(v)\ud v \right)\ud z,
\end{equation}
and  prove that $\int_{(0,\infty)}(1 \wedge z)\Pi(\ud z)$ is finite. To this aim, we observe  that  the following  three  inequalities hold true,
\begin{equation} \label{eq_3facts}
\int^{\infty}\frac{\bar\mu(w)}{w}\ud w<\infty, \qquad \int_0\bar\mu(w)\ud w <\infty \quad \text{and} \quad
\int_u^{\infty}\frac{e^{-Kz}}{z}\ud z \leq \frac{e^{-Ku}}{Ku}.
\end{equation}
Indeed, the finiteness of the first two integral follows from Fubini-Tonelli's theorem,  since
\[
\int_1^{\infty}\frac{\bar\mu(w)}{w}\ud w=\int_1^\infty \ln(z)\mu(z) \ud z\qquad \textrm{and} \qquad \int_0^1\bar\mu(w)\ud w=\int_0^\infty (1\wedge z)\mu(z) \ud z.
\] 
With this in mind, we observe
\[
\int_0^{1} z \Pi(\ud z)\leq \frac{2}{\sigma^2} \left(\delta  +e^K \int_0^{1}\bar{\mu}(v)\ud v \right)<\infty.
\]
Moreover, 
\begin{equation*}
\begin{aligned}
\int_1^{\infty}\Pi(\ud z) &= \frac{2}{\sigma^2}\left(\delta\int_1^{\infty} \frac{e^{-Kz}}{z} \ud z +\int_1^{\infty}\frac{e^{-Kz}}{z} \int_0^z e^{Kv}\bar\mu(v)\ud v\ud z\right)\\
&\leq \frac{2}{\sigma^2}\left( \frac{\delta}{K}e^{-K} +\int_0^{\infty}e^{Kv}\bar\mu(v)\int_{v\vee 1 }^{\infty}\frac{e^{-Kz}}{z}\ud z \ud v \right)\\
&\leq \frac{2}{\sigma^2}\left( \frac{\delta}{K}e^{-K} +\frac{1}{K}\int_0^1\bar\mu(v)\ud v+\frac{1}{K}\int_1^{\infty}\frac{\bar\mu(v)}{v} \ud v \right)<\infty.
\end{aligned}
\end{equation*}
In other words, the probability measure $\nu$ is infinitely divisible with support on $(0,\infty)$ and with Laplace exponent  $-m$. Finally, if $\bar\mu(0)\leq b$, a simple computation guarantees that $k$ defined by
\[
k(z)=\frac{2 e^{-K z}}{\sigma^2 }\left(\delta+\int_0^z e^{K v}\bar{\mu}(v)\ud v \right),
\]
 is non-increasing and Theorem 15.10 in Sato~\cite{Sa} implies the self-decomposability of $\nu$.
\end{proof}

\begin{proof}[Proof of Theorem \ref{prop_casesubordinator}] 
Recall from  Theorem 3 in \cite{DFM}  that the solution to \eqref{CBI} is recurrent if and only if $\mathcal{I}=\infty$.
From the definition of functions $\phi$ and $\omega$ and the fact that $2\delta\geq \sigma^2$ and $c>0$, we deduce that $\mathcal{I}=\infty$ if and only if \eqref{cond_rec} is satisfied. 

In other words, under the assumption that  $2\delta \geq \sigma^2$ and \eqref{cond_rec} hold, we have that   $0$ is polar and that $R$ is recurrent. From  the proof of  \eqref{eq_limeta}, we deduce
 that $C_t$,  the right inverse of $\eta_t$, is well defined on $(0,\infty)$ and that $Z_t=R_{C_t}$ for any $t\geq 0$. That is to say $Z$ is also recurrent, $T_0=\infty$ a.s. and  has an invariant measure that we denote  by $\rho$.

Next, we characterise the  invariant measure  $\rho$ below. In order to do so, we use  the infinitesimal generator $\U$ of $Z$, i.e. $\rho$ is an invariant measure for $Z$ if and only if 
\[
\int_0^\infty \U f(z) \rho(\ud z)=0,
\]
for any $f$ in the domain of $\U$.  According to  Palau and Pardo \cite{PP}, the infinitesimal generator $\U$ satisfies  for any $f \in C^2_b(\R_+)$,
 \[
 \U f(x)=x \A f(x) -cx^2f'(x)+\frac{\sigma^2}{2}x^2f''(x),
\]
 where $\A$ represents the generator of the spectrally positive L\'evy process associated to the branching mechanism $\psi$. For the particular choice of $f(x)=e^{-\lambda x}$, for $\lambda>0$, we observe   $\A f(x)=\psi(\lambda)e^{-\lambda x}$ implying that 
\[
0=  \int_0^{\infty} \U f(z) \rho(\ud z) =\int_0^{\infty} \Big(\psi(\lambda)+\omega(\lambda)z\Big)z e^{-\lambda z}\rho(\ud z).
\]
 Then, similarly as  in~\cite{Lambert2005}, we denote the Laplace transform of $z\rho(\ud z)$ by $\chi$ and performing the previous identity, we observe that $\chi$ satisfies the ordinary differential equation $\psi(\lambda) \chi(\lambda)-\omega(\lambda) \chi'(\lambda)=0$ on $(0,\infty)$. Straightforward computations implies that $\chi$ satisfies
 \begin{equation}\label{formulachi}
 \chi(\lambda)=K_0 \exp\left\{\int_\theta^{\lambda}\frac{\psi(u)}{\omega(u)}\ud u\right\},
 \end{equation}
  for some constants $K_0>0$ and $\theta\geq 0$. We can now prove the cases (a) and (b).

 Let us assume that \eqref{logcondition} is satisfied or   equivalently  the integrability of $\psi/\omega$ at $0$. We take  $\theta=0$ in~\eqref{formulachi} and deduce that  $\chi(\lambda)=K_0 \exp(m(\lambda))$ for some constant $K_0>0$. In other words, we have for $z\ge 0$
 \[
  \rho(\ud z)=K_0\frac{1}{z} \nu(\ud z),
 \]
with a possible Dirac mass at $0$, and where $\nu$ is defined in Lemma~\ref{lemma_mlambda}. We can conclude as soon as we prove that $\varrho:=\int_0^{\infty}z^{-1}\nu(\ud z)$ is finite if $2\delta>\sigma^2$ or if $2\delta=\sigma^2$ and condition $(\partial)$ holds and that $\rho$ is infinite if $2\delta=\sigma^2$ and condition $(\eth)$ holds. Indeed, if $\varrho<\infty$, $\rho$ defined by~\eqref{def_rho} is the unique invariant probability measure of $Z$ and consequently it is positive recurrent. If $\varrho=\infty$, then all invariant measures of $Z$ are non-integrable at $0$, so that $Z_t$ converges to $0$ in probability and since $Z$ oscillates in $(0,\infty)$ then it is null-recurrent.

Therefore, it remains to verify whether $\varrho$ is finite or not. Note that formally,
 $$
\int_0^{\infty}e^{m(\lambda)} \ud\lambda= \int_{(0,\infty)} z^{-1}\nu(\ud z)=\varrho .
 $$
 Hence, $\varrho$ is finite if and only if $e^{m(\lambda)}$ is integrable at $\infty$.
 From the proof of Lemma~\ref{lemma_mlambda} (see \eqref{eq_m} and \eqref{eq_m2ndterm}), we deduce
\begin{equation}\label{eq_mh}
 -m(\lambda)=\frac{2\delta}{\sigma^2}\ln\left(1+\frac{\lambda}{K}\right) +\int_0^{+\infty}\frac{(1-e^{-\lambda z})}{z}h(z)\ud z,
 \end{equation}
 where we recall that $K=2c/\sigma^2$, and 
 \[
 h(z)=\frac{2}{\sigma^2}e^{-Kz}\int_0^ze^{Kw}\bar\mu(w)\ud w.
 \] 
 With all this in mind, we study the integral in the right-hand side of \eqref{eq_mh} for $\lambda$ large enough following a similar approach to the proofs of Theorem 53.6 in Sato \cite{Sa} or Theorem 3.4 in Lambert~\cite{Lambert2005}. We take $x>0$ and $\lambda>1$, and split the interval  $(0,\infty)$ into $(0,x/\lambda]$, $(x/\lambda,x]$ and $(x,\infty)$. From \eqref{eq_3facts},we deduce 
\begin{align*}
\int_x^{\infty}\frac{h(z)}{z}\ud z &=\frac{2}{\sigma^2}\int_0^{\infty}e^{Kw}\bar\mu(w) \left(\int_{x\vee w}^{\infty}\frac{e^{-Kz}}{z}\ud z \right)\ud w\\
&\leq \frac{2}{K\sigma^2}\left( \frac{1}{x}\int_0^x \bar\mu(w)\ud w +\int_x^{\infty} \frac{\bar \mu(w)}{w}\ud w  \right)<\infty,
\end{align*}
which guarantees, together with the  Dominated Convergence Theorem, that 
\[
\int_x^{\infty}(1-e^{-\lambda z})\frac{h(z)}{z}\ud z \quad \textrm{converges as}\quad \lambda\to \infty.
\] 
On the other hand, we observe
\begin{align*}
\int_0^{x/\lambda} (1-e^{-\lambda z})\frac{h(z)}{z}\ud z &=\frac{\sigma^2}{2}\int_0^x \frac{(1-e^{-z})}{z}e^{-\frac{K z}{\lambda}}\left(\int_0^{z/\lambda}e^{K w}\bar\mu(w)\ud w\right)\ud z\\
&\leq \frac{\sigma^2}{2} e^{K x} \int_0^x \frac{(1-e^{-z})}{z}\ud z \int_0^{x}\bar\mu(w)\ud w<\infty,
\end{align*}
 which implies  the convergence of 
 \[
 \int_0^{x/\lambda}(1-e^{-\lambda z})\frac{h(z)}{z}\ud z\qquad  \textrm{when} \quad \lambda\to \infty.
 \]
A similar change of variables lead us to deduce 
  \[
   \int_{x/\lambda}^x e^{-\lambda z}{h(z)}/{z}\ud z 
   \]  
   converges when $\lambda$ grows to $\infty$.  Putting the pieces together in~\eqref{eq_mh}, we deduce that for any $x>0$ and for $\lambda$ large enough 
 \[
 -m(\lambda)=\frac{2\delta}{\sigma^2}\ln\left(1+\frac{\lambda}{K}\right)+\int_x^{x/\lambda}\frac{h(z)}{z}\ud z + K_1(x)+{o}(1),
 \]
 where $K_1(x)$ is a non-negative constant. Hence, for $\lambda$ large enough and for any $x>0$, 
 \begin{equation}\label{eq_approxmint}
 e^{m(\lambda)}= \frac{{K_2(x)}}{(1+\lambda)^{2\delta/\sigma^2}} \exp\left\{- \int_x^{x/\lambda}\frac{h(z)}{z}\ud z+{o}(1)\right\},
 \end{equation}
 where $K_2 (x)$ is a positive constant.
 
 It thus remains to study the integral term in \eqref{eq_approxmint}. 
 Since $h$ is positive, we can find $K_3(x)>0$ such that for any $\lambda $ large enough, $
   e^{m(\lambda)} \leq {K_3(x)}{\lambda^{-2\delta/\sigma^2}},
 $
 and we conclude as soon as $2\delta >\sigma^2$. This implies part (a), when \eqref{logcondition} holds.
 
Next, we prove part (b), i.e. we assume that $2\delta=\sigma^2$ and that \eqref{logcondition} holds. For the sake of brevity, we concentrate on the case $(\partial)$,  the case $(\eth)$ uses similar arguments. 
Under condition $(\partial)$, there exists $n\in \N $ such that $\inf ({\tt Adh}(I^{(n)}))>{\sigma^2}/{2}$ and ${\tt Adh}(I^{(k)})=\{{\sigma^2}/{2}\}$, for any $k\in\{1,..,n-1\}$. Let us define by recurrence the collection of functions $\bar I$ such that
 \[
 \bar I^{(1)}(z)=l^{(1)}(z)h(z)\qquad  \text{ and }\qquad \bar I^{(k)}(z)= l^{(k)}(z)\left[\bar I^{(k-1)}(z)-1\right],\quad \; k\in \N,\,\, k\geq 2.
 \]
 Note that the sequences $\{\bar I^{(k)}\}_{k\leq n}$ and $\{I^{(k)}\}_{k\leq n}$ satisfy similar recurrence relations  but start on different values.
 From the definition of $h$ and a recurrence argument, it is straightforward to compute that for any $k\in \N$, and for $z$ small enough,
\[
\frac{2}{\sigma^2} e^{-Kz} I^{(k)}(z) +(e^{-Kz}-1)\sum_{j=2}^k\prod_{i=j}^k l^{(i)}(z) \leq \bar I^{(k)}(z) \leq  \frac{2}{\sigma^2}  I^{(k)}(z).
\] 
Since $(e^{-Kz}-1)$ behaves as $-Kz$,  for $z$ small enough, the second term of the left hand side converges to $0$ when $z$ converges to $0$ and we deduce that the sequences of functions $\{\bar I^{(k)}\}_{k\leq n}$ and $\{I^{(k)}\}_{k\leq n}$ satisfy similar assumptions, which are $\inf({\tt Adh}(\bar I^{(n)}))=A>1$ and ${\tt Adh}(\bar I^{(k)})=\{1\}$, for any $k\in\{1,..,n-1\}$. Let us fix $\varepsilon>0$ such that $A-\varepsilon>1$ and  $x>0$ such that $\bar I^{(n)}(x)\geq A-\varepsilon$. Using the definition of $\{\bar I^{(k)}\}_{k\geq 0}$ and a  recurrence argument, we obtain that for any $z$ sufficiently small,
\[
h(z)=\frac{\bar I^{(n)}(z)}{\prod_{i=1}^n l^{(i)}(z)}+\sum_{j=1}^{n-1} \frac{1}{\prod_{i=1}^j l^{(i)}(z)}.
\]
Hence,
\begin{equation}\label{eq_intxxlambda}
\int_{x/\lambda}^{x}\frac{h(z)}{z}\ud z \geq (A-\varepsilon) \int_{x/\lambda}^x \frac{\ud z}{z \prod_{i=1}^nl^{(i)}(z)} + \sum_{j=1}^{n-1}\int_{x/\lambda}^x\frac{\ud z}{z\prod_{i=1}^{j}l^{(i)}(z)}. 
\end{equation}
Moreover  from the definition of $l^{(j)}$, we have for any $j\in \N$,
\begin{align*}
\int_{x/\lambda}^x\frac{\ud z}{z\prod_{i=1}^{j}l^{(i)}(z)}&=l^{{(j+1)}}(x)-l^{{(j+1)}}\left(\frac{x}{\lambda}\right)\\
&=l^{(j+1)}(x)+l^{(j+1)}(\lambda)-R^{(j+1)}(x,\lambda) \qquad  \textrm{as}\quad \lambda \to \infty,
\end{align*}
where the sequence $\{R^{(k)}\}_{k\geq 2}$ satisfies the following recurrence relation: for any $x$ small enough and $\lambda$ large enough,
\[
 R^{(2)}(x,\lambda)=\ln\left(1+\frac{l^{(1)}(x)}{l^{(1)}(\lambda)} \right)\qquad  \text{ and }\qquad R^{(j)}(z)= \ln\left(1+\frac{R^{(j-1)}(x,\lambda)}{l^{(j-1)}(\lambda)} \right),\quad \; j\in \{3,..,n+1\}.
\]
Hence, we deduce that $R^{(j)}(x,\lambda)$ converges to $0$ when $\lambda$ increases to $\infty$, for all $j\in \{3,..,n+1\}$.
In addition with~\eqref{eq_intxxlambda}, as soon as $\lambda$ is sufficiently large, we have
$$
\int_{x/\lambda}^{x}\frac{h(z)}{z}\ud z \geq (A-\varepsilon) l^{(n+1)}(\lambda) + \sum_{j=1}^{n-1}l^{(j+1)}(\lambda)+ K_4(x), 
$$
where $K_4(x)$ is a finite constant. Hence using~\eqref{eq_approxmint}, we deduce that for $\lambda$ sufficiently large there exist a finite constant $K_5(x)>0$ such that 
\begin{equation}\label{eq_approxmmaj}
e^{m(\lambda)}\leq  \frac{{K_5(x)}}{\lambda \displaystyle \prod_{i=1}^{n-1} l^{(i)}(\lambda)  (l^{(n)}(\lambda))^{A-\varepsilon} }.
\end{equation}
Since $A-\varepsilon>1$, the right hand side of~\eqref{eq_approxmmaj} is integrable at $\infty$. Indeed, for any $z,y$ sufficiently large such that $l^{(n)}(y)>0$ and $l^{(n)}(z)>0$, with the change of variables $u=l^{n}(\lambda)$, we have
$$
\int_z^y  \frac{1}{\lambda \displaystyle \prod_{i=1}^{n-1} l^{(i)}(\lambda)  (l^{(n)}(\lambda))^{A-\varepsilon}  } \ud \lambda= \int_{l^{(n)}(z)}^{l^{(n)}(y)}\frac{1}{u^{A-\varepsilon}}\ud u \underset{b\to \infty}{\longrightarrow}\int_{l^{(n)}(z)}^{\infty}\frac{1}{u^{A-\varepsilon}}\ud u<\infty.
$$
Finally, we have proved that under condition $(\partial)$, \[
\int^{\infty}e^{m(\lambda)}\ud \lambda<\infty.
\] 
 This completes  the proof of part (b) and the cases when  condition~\eqref{logcondition} is satisfied.

Now,  we deal with the case when the log-moment condition \eqref{logcondition} does not hold and $2\delta >\sigma^2$. Under this assumption we show that  $Z$ is still positive recurrent  but its invariant distribution has an infinite expected value. Recall that condition $2\delta >\sigma^2$ guarantees  that $Z$ is recurrent with an invariant distribution $\rho$ satisfying~\eqref{formulachi}. However in this case, $\psi/\omega$ is not integrable at $0$ and we can not take $\theta=0$ in ~\eqref{formulachi}, instead we let $\theta=1$. Formally, the following identity still holds
\[
 \int_0^{\infty}\chi(\lambda)\ud \lambda=\int_0^{\infty}\rho(\ud z).
\]
Our aim is thus to prove that the latter identity is finite but the expected value of $\rho$ is infinite.

On the one hand, recalling that $K=2c/\sigma^2$ and taking $\lambda$  smaller than $1$, we  use the definition of $\psi$ and Fubini-Tonnelli's Theorem to deduce
\begin{align*}
-\int_{\lambda}^1\frac{\psi(z)}{\omega(z)}\ud z&= \frac{2\delta}{\sigma^2}\ln\left(\frac{K+1}{K+\lambda}\right)+\frac{2}{\sigma^2}\int_0^{\infty}\left(\int_{\lambda}^1\frac{1-e^{-zu}}{Kz+z^2}\ud z\right)\mu(\ud u)\\
& \leq \frac{2\delta}{\sigma^2}\ln\left(1+\frac{1}{K}\right)+\frac{2}{\sigma^2}\int_0^{A}\left(\int_{\lambda}^1\frac{zu}{Kz}\ud z\right)\mu(\ud u)+\frac{2}{\sigma^2}\int_A^{\infty}\left(\int_{\lambda}^1\frac{1}{Kz}\ud z\right)\mu(\ud u)\\
& \leq \frac{2\delta}{\sigma^2}\ln\left(1+\frac{1}{K}\right)+\frac{2}{K\sigma^2}\int_0^{A}u\mu(\ud u)-\frac{2}{K\sigma^2}\ln(\lambda)\bar\mu(A),
\end{align*}
for any $A>0$.  Thus,  we take $A>0$ in such a way  that $\bar\mu(A)\leq {K}{\sigma^2}/4$. Implying that  for any $\lambda\leq 1$, we get
\[
\chi(\lambda) \leq K_0 \frac{e^{K(A)}}{\lambda^{1/2}},
\]
with $K_0$ and $K(A)$ two positive constants which are  independent from $\lambda$. In other words, $\chi$ is integrable near $0$. On the other hand, since 
\[
\int_{1}^{\lambda}\frac{\psi(z)}{\omega(z)}\ud z\leq -\frac{2b}{\sigma^2}\ln\left(\frac{K+1}{K+\lambda}\right),
\] we also have
\[
 \chi(\lambda)\leq K_0 \left(\frac{K+1}{K+\lambda}\right)^{\frac{2b}{\sigma^2}},
\]
implying that 
\[
\int_0^{\infty}\chi(\lambda)\ud \lambda<\infty,
\]
since $2b>\sigma^2$. In other words $Z$ has a finite invariant measure and is positive recurrent. Moreover, since the log-moment condition~\eqref{logcondition} does not hold, a straightforward computation gives
\[
\int_0^{\infty}z\rho(\ud z)=\lim_{\lambda\to 0} \int_0^{\infty}e^{-\lambda z} z \rho(\ud z) =\lim_{\lambda\to 0} \chi(\lambda)=\infty.
\] 

Finally, if  condition \eqref{cond_rec}  does not hold then $\mathcal{I}<\infty$ and from Theorem \ref{teo1} the process $Z$ explodes in finite time a.s. 
\end{proof}


\section{General case}
\label{subsec_logisticLaplace}

For the proof of Theorem  \ref{theo_Tatteinte} recall that  the associated L\'evy process $X$ which appears in \eqref{SDE1} is general, that is to say, there exist $\zmin\geq0$ such that $\psi(z)>0$ for any $z\geq \zmin$ and the log-moment condition \eqref{logcondition}
is satisfied. 

\begin{proof}[Proof of Theorem \ref{theo_Tatteinte}]

Let us fix $\lambda>0$, and denote by $\func$ the function
\begin{equation*}
\func(z):= \frac{e^{-m(z)}}{\omega(z)} \exp\left\{ -\int_0^{{\tt I}(z)}y_\lambda(v)\ud v \right\} \int_0^z \exp\left\{ m(u)+2\int_0^{{\tt I}(u)}y_\lambda(v)\ud v \right\}\ud u,
\end{equation*}
in other words, we have 
\[
h_\lambda(x)=1+\lambda \int_0^{\infty}e^{-xz}\func (z) \ud z,
\]
which was defined by \eqref{eq-theo}. For simplicity in exposition, we split the proof in six steps.

{\bf Step 1:} We first prove that $h_\lambda$ is well defined on $(0,\infty)$ or equivalently, we prove that $z\mapsto e^{-xz}\func (z)$ is integrable on $(0,\infty)$ as soon as $x$ is positive. From the definitions of $m$ and  ${\tt I}$ (see \eqref{defm} and \eqref{defI}, respectively), it is straightforward that
\begin{equation}
 \label{eq_conv1}
 \exp\left\{m(u)+2\int_0^{{\tt I}(u)} y_\lambda(v) \ud v \right\}  \to 1, \qquad \textrm{as}\quad u\to 0,
\end{equation}
implying
\begin{equation}\label{eq_conv2}
e^{-xz}\func (z) {\sim} \frac{z}{\omega(z)} \sim \frac{1}{c},  \qquad \textrm{as}\quad z\to 0,
\end{equation}
hence the integrability at $0$.

Concerning the neighbourhood of $\infty$, we see from   Lemma \ref{lemma_hq} that $y_\lambda(z)\leq \sqrt{\lambda} \frac{\varphi'(z)}{\sqrt{\omega(\varphi(z))}}$ which is equivalent to $\sqrt{2\lambda}\frac{\varphi'(z)}{\sigma \varphi(z)}$. In addition with \eqref{eq_derphi}, we deduce
\begin{equation}
\label{eq_convinty}
 \int_0^{{\tt I}(z)}y_\lambda(u) \ud u = {\rm O}\left( \ln (z) \right) \qquad \text{and} \qquad \frac{\psi(z)}{\omega(z)}+ {\tt I}'(z)y_\lambda({\tt I}(z))\geq 0 \qquad \textrm{as} \quad z\to \infty.
\end{equation}
Then, for any $x>0$ and for $u$ sufficiently large, we have
\[
\left|\frac{\exp\left\{m(u)+2\int_0^{{\tt I}(u)} y_\lambda(v) \ud v  \right\} }{\left(\frac{x}{2}+\frac{\psi(u)}{\omega(u)}+{\tt I}'(u)y_\lambda({\tt I}(u))\right)\exp\left\{\frac{xu}{2}+m(u)+\int_0^{{\tt I}(u)} y_\lambda(v) \ud v  \right\} }\right|\leq  \frac{2\exp\left\{-\frac{xu}{2}+\int_0^{{\tt I}(u)} y_\lambda(v) \ud v  \right\} }{x},
\]
which converges to 0 as $u$ goes to $\infty$. In other words, 
\begin{equation}\label{eq_oitoint}
\int_0^z  \exp\left\{m(u)+2\int_0^{{\tt I}(u)} y_\lambda(v) \ud v  \right\} \ud z = {\rm o}\left(\exp\left\{\frac{xz}{2}+m(z)+\int_0^{{\tt I}(z)} y_\lambda(v) \ud v  \right\} \right), \quad \textrm{as}\quad z\to \infty.
\end{equation}
Finally from the definition of $\func $, we obtain
\begin{equation}
\label{eq_oitog}
e^{-zx}\func (z)={\rm o}\left(\frac{1}{\omega(z)} e^{-\frac{xz}{2}} \right), \qquad \textrm{as} \quad z \to \infty,
\end{equation}
implying the integrability of $z\mapsto e^{-zx}\func (z)$ at $\infty$. It is important to note that  \eqref{eq_conv2} and \eqref{eq_oitog}, also imply that  the mappings $z\mapsto ze^{-xz}\func (z)$ and $z\mapsto z^2e^{-zx}\func (z)$ are integrable on $(0,\infty)$ and that $h_\lambda$ is a $C^2$-function on $(0,\infty)$.

{\bf Step 2:} Now, we prove \eqref{eq_vpgenerator}. The infinitesimal generator of $Z$ satisfies \eqref{infgenerator}, i.e.  for any $f \in C^2_b(\R_+)$
 \begin{equation}\label{eq_generator}
  \U f(x)=x \A f(x) -cx^2f'(x)+\frac{\sigma^2}{2}x^2f''(x)
  \end{equation}
where  $\A$ is the generator of the spectrally positive L\'evy process associated to branching mechanism $\psi$. 
Since, for $f(x)=e^{-z x}$,  $\A f(x)=\psi(z)e^{-zx}$ with $z\ge 0$, we deduce using integrations by parts (twice) that
\begin{equation}
\label{eq_calcul}
\begin{aligned}
\U h_\lambda(x)-\lambda h_\lambda(x)&= \lambda\int_0^{\infty}\Big(x\psi(z)+x^2\omega(z)-\lambda \Big)\func (z)e^{-zx}\ud z - \lambda\\
&=\lambda\bigg( \int_0^{\infty} \Big((\psi \func )'(z)+(\omega \func )''(z)-\lambda \func (z)\Big)e^{-xz}\ud z -1\\
&\quad -xw(z)\func (z)e^{-xz}\bigg|_{z=0}^{z=\infty}+\Big(\psi(z)\func (z)+(\omega \func )'(z)\Big)e^{-xz}\bigg|_{z=0}^{z=\infty}\bigg).
\end{aligned}
\end{equation}
Let us prove that the right-hand side of the latter expression equals 0. Recall that $m'(z)=\frac{\psi(z)}{\omega(z)}$ and ${\tt I}'(z)=e^{m(z)}$, then
\begin{equation}
\label{eq_wg'}
(\omega \func )'(z)=-\psi(z)\func (z)-y_\lambda({\tt I}(z))e^{-\int_0^{{\tt I}(z)}y_\lambda(v) \ud v }\int_0^z e^{m(u)+2\int_0^{{\tt I}(u)}y_\lambda(v) \ud v }\ud u + e^{\int_0^{{\tt I}(z)}y_\lambda(v) \ud v }.
\end{equation}
In addition with the fact that  $y_\lambda$ is solution to \eqref{eq_h}, we deduce that $(\omega \func )''(z)=-(\psi \func )'(z)+\lambda \func (z)$ for any $z\geq 0$. 
 On the other hand, using \eqref{eq_conv2} and \eqref{eq_oitog}, we have that 
 \[
 xw(z)\func (z)e^{-xz}\bigg|_{z=0}^{z=\infty}=0,
 \]
  and from \eqref{eq_wg'}, together with \eqref{eq_convinty} and \eqref{eq_oitoint}, we deduce 
  \[
  \lim_{z\to \infty} (\psi(z)\func (z)+(\omega \func )'(z))e^{-xz}=0,
  \]  
  as soon as $x>0$. Therefore, it remains to study the previous limit  but when $z$ goes to $0$. According to \eqref{eq_wg'},
 \begin{equation}
 \label{eq_limwg}
 \lim_{z\to 0} (\psi(z)\func (z)+(\omega \func )'(z))e^{-xz}=1-\lim_{z\to 0} y_\lambda({\tt I}(z))\int_{0}^z e^{m(u)+2\int_0^{{\tt I}(u)}y_\lambda(v) \ud v }\ud u.
 \end{equation}
  By Lemma \ref{lemma_hq} and \eqref{eq_conv1}, we deduce
\[
 y_\lambda({\tt I}(z))\int_0^z e^{m(u)+2\int_0^{{\tt I}(u)}y_\lambda(v) \ud v }\ud u \leq \sqrt{\frac{\lambda}{\omega(z)}}e^{-m(z)}\int_0^z e^{m(u)+2\int_0^{{\tt I}(u)}y_\lambda(v) \ud v }\ud u \sim \sqrt{\frac{\lambda}{cz}}z, \quad\textrm{as} \quad z\to 0,
\]
 which implies that the right-hand side of \eqref{eq_limwg} equals 1. In other words,  the right-hand side of \eqref{eq_calcul} equals 0, meaning that $\U h_\lambda(x)=\lambda h_\lambda(x)$ for any $x>0$ and that \eqref{eq_vpgenerator} holds.

{\bf Step 3:} Our next step is to prove that $\int_0^{\infty}y_\lambda(v) \ud v $ is finite as soon as $\mathbb{P}_x(T_0<\infty)=1$, for any $x>0$, 
actually Lemma \ref{lemma_hq} is not enough to conclude. With this goal in mind, we fix $x>0$ and $\lambda\geq 0$ and set the function $\funcG_{\lambda,x} $ as follows, 
 \begin{equation*}
  \funcG_{\lambda,x} (v):=\int_0^{\infty}e^{-\lambda t}\E_x\Big[e^{-v Z_t}\Big]\ud t, \qquad \textrm{for any}  \quad v\geq 0.
 \end{equation*}
 This function is related with the Laplace transform of  $T_0$, indeed 
\[
 \lim_{v\to \infty} \lambda \funcG_{\lambda,x} (v)=\E_x\Big[e^{-\lambda T_0}\Big].
 \] 
The latter is positive since $\mathbb{P}_x(T_0<\infty)=1$. Our aim is to find a second formulation to $\funcG_{\lambda,x}$, related to $\int_0^{\infty}y_\lambda(v) \ud v $, to conclude. 

Let us provide some properties of $\funcG_{\lambda,x}$. We first  note that for any $h$ belonging to the domain of $\U$, the following identity holds
\[
 \lambda \int_0^{\infty}e^{-\lambda t}\E_x\Big[h(Z_t)\Big]\ud t= h(x) + \int_0^{\infty}e^{-\lambda t}\E_x\Big[\U h(Z_t)\Big]\ud t.
\]
By taking  $h(x)=e^{-vx}$ together with  identity  \eqref{eq_generator}, we deduce
 \begin{equation*}
 \begin{aligned}
 \lambda \funcG_{\lambda,x} (v)&=e^{-v x}+\int_0^{\infty} e^{-\lambda t} \E_x\left[\psi(v)Z_s e^{-v Z_s}+\omega(v)Z_t^2e^{-v Z_t}\right]\ud t\\
 &=e^{-vx}-\psi(v) \funcG^{\prime}_{\lambda,x} (v)+\omega(v) \funcG^{\prime\prime}_{\lambda,x} (v).
 \end{aligned}
\end{equation*}
 Moreover $\lambda \funcG_{\lambda,x}  (0)=1$ and the dominated convergence theorem implies
 \[
 \funcG^{\prime}_{\lambda,x} (v)= -\int_0^{\infty} e^{-\lambda t} \E_x[Z_t e^{-v Z_t}1_{\{Z_t>0\}}]\ud t \longrightarrow 0, \quad\textrm{as} \quad v\to \infty.
\]
 We  now prove that $\funcG_{\lambda,x} $ is the unique solution to $\omega(v) y''(v) - \psi(v) y'(v) -\lambda y(v)=e^{-vx}$ with conditions $\lambda y(0)=1$ and $\lim_{v\to\infty}y'(v)= 0$. In order to do so, we will explicit the set of functions that satisfy the equation with condition $\lambda y(0)=1$. First of all, let us prove that the following function, for any $v\geq 0$,
\begin{equation}
\label{def_funck}
\funck (v):=\frac{1}{\lambda}e^{-\int_0^{{\tt I}(v)}y_\lambda(s) \ud s } \left( 1+\lambda\int_0^v \int_u^{\infty} \frac{e^{-zx}}{\omega(z)}e^{-m(z)-\int_0^{{\tt I}(z)}y_\lambda(s) \ud s +m(u)+2\int_0^{{\tt I}(u)}y_\lambda(s) \ud s } \ud z \ud u \right)
\end{equation}
satisfies the same conditions as $\funcG_{\lambda,x} $.
We first observe that 
\begin{equation}\label{fubini}
\begin{split}
 \int_0^v \int_u^{\infty}& \frac{e^{-zx}}{\omega(z)}e^{-m(z)-\int_0^{{\tt I}(z)}y_\lambda(s) \ud s +m(u)+2\int_0^{{\tt I}(u)}y_\lambda(s) \ud s } \ud z \ud u \\
 &\hspace{3cm}= \int_0^{\infty} \frac{e^{-zx}}{\omega(z)}e^{-m(z)-\int_0^{{\tt I}(z)}y_\lambda(s) \ud s }\left( \int_0^{v\wedge z} e^{m(u)+2\int_0^{{\tt I}(u)}y_\lambda(s) \ud s } \ud u\right) \ud z
 \end{split}
\end{equation}
is finite according to~\eqref{eq_conv1}. In other words, $\funck$ is well defined.
Moreover, $\lambda \funck (0)=1$ and since ${\tt I}^{\prime}(z)=\exp(m(z))$, a straightforward computation gives
\begin{equation}
\label{eq_derk}
\funck^{\prime}(v)=-e^{m(v)}y_\lambda({\tt I}(v))\funck (v)+e^{m(v)+\int_0^{{\tt I}(v)} y_\lambda(s) \ud s} \int_v^{\infty} \frac{e^{-zx}}{\omega(z)}e^{-m(z)-\int_0^{{\tt I}(z)}y_\lambda(s) \ud s} \ud z.
\end{equation}
From~\eqref{eq_oitoint} and~\eqref{fubini}, we deduce that $ \funck $ is bounded by some constant $C$ on $\R$ and from Lemma \ref{lemma_hq}, we also see  that 
\[
\left| e^{m(v)}y_\lambda({\tt I}(v))\funck (v) \right| \leq C \sqrt{\frac{\lambda}{\omega(v)}} \longrightarrow 0, \qquad \textrm{as}\quad v\to+\infty.
\]
For the second term of the right-hand side of \eqref{eq_derk}, we use a similar arguments to those  used to deduce  \eqref{eq_oitoint} which gives
\[
\int_v^{\infty} \frac{e^{-xz}}{\omega(z)}e^{-m(z)-\int_0^{{\tt I}(z)}y_\lambda(s) \ud s} \ud z = {o} \left( e^{-m(v)-\int_0^{{\tt I}(v)}y_\lambda(s) \ud s-\frac{xv}{2}}\right), \qquad \textrm{as}\quad v\to\infty.
\]
That is to say that $\funck^{\prime}(v)$ converges to $0$ when $v$ goes to $\infty$. Finally, from~\eqref{eq_derk}, a straightforward computation provides
\begin{equation}
\label{eq_k}
\omega(v) \funck^{\prime\prime}(v)=\psi(v)\funck^{\prime}(v)+\lambda \funck (v)-e^{-vx}.
\end{equation}
Putting all pieces together, we prove that $\funck$ and $\funcG_{\lambda, x}$ satisfy the same differential equation with conditions $\lambda k(0)=1$ and $\lim_{v\to\infty}k'(v)= 0$.

Furthermore, from this, we deduce that the set of functions that satisfy $\omega(v) y''(v) - \psi (v)y'(v) -\lambda y(v)=e^{-vx}$ with conditions $\lambda y(0)=1$ is exactly $S:=\{\funck_A, A\in \R\}$, with
\[
 \funck_A(v) := \funck(v) + A  e^{-\int_0^{{\tt I}(v)}y_\lambda(s)\ud s } \int_0^v e^{m(u)+2 \int_0^{{\tt I}(u)}y_\lambda(s)\ud s} \ud u,
\]
Let us prove that $\lim_{v\to+\infty}\funck_A'(v)= 0$ if and only if $A=0$. Indeed,
\[
 \funck_A^\prime(v)=\funck^\prime(v)+ A e^{m(v)+\int_0^{{\tt I}(v)}y_\lambda(s)\ud s} \left[1- \frac{1}{\alpha(v)} \int_0^{v} e^{m(u)+2\int_0^{{\tt I}(u)}y_\lambda(s)\ud s}\ud u \right],
 \]
 where
  \[
   \frac{1}{\alpha(v)}:=y_\lambda({\tt I}(v))e^{-2\int_0^{{\tt I}(v)}y_\lambda(s)\ud s}.
\]
Using Lemma~\ref{lemma_hq}, we have
\[
 \alpha^\prime(v) {e^{-m(v)-2\int_0^{{\tt I}(v)}y_\lambda(s)\ud s}}=-\frac{y_\lambda^\prime({\tt I}(v))}{y_\lambda({\tt I}(v))^2} + 2=1+\lambda \frac{r^2({\tt I}(v))}{y_\lambda^2({\tt I}(v))}\geq 2,
\]
for any $v$ large enough.
In other words, there exist $v_0>0$ such that for any $v\geq v_0$,
\[
 \frac{1}{\alpha(v)} \int_0^{v} e^{m(u)+2\int_0^{{\tt I}(u)}y_\lambda(s)\ud s}\ud u  \leq \frac{1}{2}+\frac{1}{\alpha(v)} \int_0^{v_0} e^{m(u)+2\int_0^{{\tt I}(u)}y_\lambda(s)\ud s}\ud u .
\]
Since $\lim_{v\to\infty}\alpha(v)=\infty$, the latter inequality guarantees
\[
 \limsup_{v\to \infty} \frac{1}{\alpha(v)} \int_0^{v} e^{m(u)+2\int_0^{{\tt I}(u)}y_\lambda(s)\ud s}\ud u \leq \frac{1}{2}.
\]
In addition to the expression of $\funck_A^\prime$, we deduce that $\lim_{v\to\infty}\funck_A'(v)= 0$ if and only if $A=0$. Thus there exist a unique function in $S$ that satisfies $\lim_{v\to\infty}\funck_A'(v)= 0$. Finally, since both $\funck$ and $\funcG_{\lambda, x}$ belong to $S$ and satisfy the condition, then both functions are equals on $\R$.

Furthermore, with a direct application of Fubini's theorem
\[
 \lim_{v\to \infty}  \int_0^v \int_u^{\infty} \frac{e^{-zx}}{\omega(z)}e^{-m(z)-\int_0^{{\tt I}(z)}y_\lambda(s)\ud s+m(u)+2\int_0^{{\tt I}(u)}y_\lambda(s)\ud s} \ud z \ud u =\int_0^{\infty}e^{-zx} \func (z)\ud z>0.
 \]
In addition with \eqref{def_funck}, we get
 \[
  e^{-\int_0^{\infty}y_\lambda(s)\ud s} \left( 1+\lambda\int_0^{\infty}e^{-zx} \func (z)\ud z \right)= \lim_{v\to\infty}\lambda \funck (v)=\lim_{v\to\infty}\lambda \funcG_{\lambda, x} (v)=\E_x\Big[e^{-\lambda T_0}\Big]>0.
 \]
We conclude that $\int_0^{\infty}y_\lambda(v)\ud v$ is finite and
 \begin{equation}
 \label{eq_formula0}
 \E_x\Big[e^{-\lambda T_0}\Big]= e^{-\int_0^{\infty}y_\lambda (v)\ud v} \left( 1+\lambda\int_0^{\infty} e^{-zx}\func (z)\ud z \right).
 \end{equation}

{\bf Step 4:} We next prove that $h_\lambda(0)=\exp\{{\int_0^{\infty}y_\lambda}(v)\ud v\}$. The main issue comes from the fact that we can not make $x$ tend to $0$ directly in the formula of $h_\lambda$ since we do not know the integrability of $\func (z)$ near $\infty$. However, from \eqref{eq_conv1} we know that for any $v\in(0,\infty)$,
\[
\lambda \int_0^{\infty} \frac{1}{\omega(z)}e^{-m(z)-\int_0^{{\tt I}(z)}y_\lambda(s)\ud s} \int_0^{z\wedge v} e^{m(u)+2\int_0^{{\tt I}(u)}y_\lambda(s)\ud s} \ud u \ud z <\infty.
\]
The goal is to take $v$ near $\infty$. Using Fubini's theorem and twice the following  change of variables $z\mapsto {\tt I}(z)$, we find
\[
\begin{split}
\lambda \int_0^{\infty}& \frac{1}{\omega(z)}e^{-m(z)-\int_0^{{\tt I}(z)}y_\lambda(s)\ud s} \int_0^{z\wedge v} e^{m(u)+2\int_0^{{\tt I}(u)}y_\lambda(s)\ud s} \ud u \ud z \\
&\hspace{4cm}= \int_0^{{\tt I}(v)}e^{2\int_0^u y_\lambda(s)\ud s} \int_u^{\infty}\lambda \frac{e^{-2m(\varphi(z))}}{w(\varphi(z))}e^{-\int_0^z y_\lambda(s)\ud s} \ud z \ud u.
\end{split}
\]
Recalling that, according to Lemma~\ref{lemma_hq}, 
\[
\lambda\frac{e^{-2m(\varphi(z))}}{w(\varphi(z))}=\lambda\frac{\varphi'(z)^2}{w(\varphi(z))}= y_{\lambda}^2(z)-y_\lambda'(z),
\]
and using integration by parts on the term $y_\lambda^2(z)e^{-\int_0^zy_\lambda}$, we finally deduce
\[
\lambda \int_0^{\infty} \frac{1}{\omega(z)}e^{-m(z)-\int_0^{{\tt I}(z)}y_\lambda(s)\ud s} \int_0^{z\wedge v} e^{m(u)+2\int_0^{{\tt I}(u)}y_\lambda(s)\ud s} \ud u \ud z = e^{\int_0^{{\tt I}(v)}y_\lambda(s)\ud s}-1.
\]
Since the integrand is positive, we let $v$ tend to $\infty$ to find
\begin{equation}
 \label{eq_flambda0}
h_\lambda(0)= 1+ \lambda \int_0^{\infty} \frac{1}{\omega(z)}e^{-m(z)-\int_0^{{\tt I}(z)}y_\lambda(s)\ud s} \int_0^{z} e^{m(u)+2\int_0^{{\tt I}(u)}y_\lambda(s)\ud s} \ud u \ud z  = e^{\int_0^{\infty}y_\lambda(s)\ud s}
\end{equation}
which is finite according to the previous step.

{\bf Step 5:} We now prove identity~\eqref{eq_resultLaplace}. First, let us assume that $x\geq a>0$ and define for any $n\in \N$, 
\[
\theta_n=\inf\{t\geq 0, Z_t \geq n\}.
\]
 Recalling that $\mathcal{U}h_\lambda=\lambda h_\lambda$ and that the expression of $\mathcal{U}$ is given by \eqref{infgenerator}, we deduce from It\^o's formula applied to $Z_{t\wedge T_a \wedge \theta_n}$ and $C^2$-function $(t,u)\mapsto e^{-\lambda t}h_\lambda(u)$ that, for any $n\in \N$,
\begin{equation}\label{fitobis}
\begin{aligned}
e^{-\lambda t\wedge T_a \wedge \theta_n}h_\lambda(Z_{t\wedge T_a \wedge \theta_n})=&h_\lambda(x)+\int_0^{t\wedge T_a\wedge \theta_n}e^{-\lambda s}h_\lambda'(Z_{s})\sqrt{2\gamma^2 Z_{s}} \ud B_s+\int_0^{t\wedge T_a\wedge \theta_n} \sigma e^{-\lambda s}h_\lambda'(Z_s) Z_s \ud B^{(e)}_s \\
&+\int_0^{t\wedge T_a\wedge \theta_n}\int_{(0,1)}\int_0^{Z_{s-}}e^{-\lambda s} \left( h_\lambda(Z_{s-}+z)-h_\lambda(Z_{s-})\right)\widetilde N^{(b)}(\ud s,\ud z,\ud u)\\
&+\int_0^{t\wedge T_a\wedge \theta_n}\int_{[1,{\infty})}\int_0^{Z_{s-}}e^{-\lambda s} \left( h_\lambda(Z_{s-}+z)-h_\lambda(Z_{s-})\right) N^{(b)}(\ud s,\ud z,\ud u)\\
&-\int_0^{t\wedge T_a\wedge \theta_n}e^{-\lambda s} Z_s \int_{[1,{\infty})}\  \left( h_\lambda(Z_{s}+z)-h_\lambda(Z_{s})\right)\mu(\ud z)\ud s,
\end{aligned}
\end{equation}
where all terms are well defined since
$h_\lambda$ is positive non-increasing, $h'_\lambda$ is negative non-decreasing, and $(Z_{s}, s \leq t\wedge T_a\wedge \theta_n)$ take values on $[a,n]$. According to the same arguments, the three first integrals of the r.h.s. of \eqref{fitobis} are martingales. Since
\[
\E\left[\int_0^{t\wedge T_a\wedge \theta_n}\int_1^{\infty}\int_0^{Z_{s}}\left|e^{-\lambda s} \left( h_\lambda(Z_{s}+z)-h_\lambda(Z_{s-})\right)\right| \mu(\ud z)\ud s\ud u\right] \leq 2 h_\lambda(a)tn \int_1^\infty \mu(\ud z)< +\infty,
\]
the fourth and integrals of the r.h.s. of \eqref{fitobis} can be written as a martingale by observing that the fifth integral is exactly the compensator. Finally,
taking the expectation of \eqref{fitobis}, we obtain for any $n\ge 1$ and $t\geq 0$ that
\[
\E_x\left[ e^{-\lambda t\wedge T_a\wedge \theta_n}h_\lambda(Z_{t\wedge T_a\wedge \theta_n}) \right] =h_\lambda(x).
\]
Since $h_\lambda$ is bounded by $h_\lambda(0)<\infty$, we use the dominated convergence theorem and make  $n$ go to $0$ and $t$ go to $\infty$. Since $\mathbb{P}_x(T_a<\infty)=1$  (recall that we are assuming that  $\mathbb{P}_x(T_0<\infty)=1$) and thus $Z_{T_a}=a$,  $\mathbb{P}_x$-a.s., we deduce \eqref{eq_resultLaplace} for $x\geq a>0$. For $a=0$, identity \eqref{eq_resultLaplace} has already been obtained in \eqref{eq_formula0} and \eqref{eq_flambda0}.

{\bf Step 6:} Next, we handle the result on the expectation of $T_0$, i.e. identity \eqref{eq_resultexp}, using similar arguments as in the proof of Theorem 3.9 in \cite{Lambert2005}. We denote $H(t,\lambda)$ for the Laplace transform $\E_{x}[e^{-\lambda Z_t}]$, and observe
$$
\lim_{\lambda\to \infty} \int_0^{\infty}(1-H(t,\lambda))\ud t =\E_x\left[\int_0^\infty \mathbf{1}_{\{Z_t>0\}} dt \right]=\E_x\Big[T_0\Big].
$$
On the other hand, from \eqref{eq_generator}, for any $t\geq 0, \lambda>0$,
$$
\frac{\partial H}{\partial t}(t,\lambda)=-\psi(\lambda)\frac{\partial H}{\partial \lambda}(t,\lambda)+\omega(\lambda)\frac{\partial^2 H}{\partial \lambda^2}(t,\lambda)=\omega(\lambda)e^{m(\lambda)}\frac{\partial}{\partial\lambda}\left( \frac{\partial H}{\partial\lambda}(t,\lambda) e^{-m(\lambda)}\right),
$$
which, by integrating $\frac{\partial H}{\partial \lambda}e^{-m(\lambda)}$ with respect to $\lambda$ yields to
$$
\frac{\partial H}{\partial \lambda}(t,\lambda)=-e^{m(\lambda)}\int_\lambda^\infty \frac{e^{-m(u)}}{\omega(u)}\frac{\partial H}{\partial t}(t,u)\ud u
$$
 and then integrating again with respect to  $\lambda$ and $t$ on $[0,\lambda]\times \R$, we obtain
$$
\int_0^{\infty}(1-H(t,\lambda))\ud t = \int_0^{\lambda} e^{m(u)}\int_u^{\infty}\frac{e^{-m(z)}}{\omega(z)}(1-e^{-zx})\ud z \ud u.
$$
Letting $\lambda$ go to $\infty$, we deduce \eqref{eq_resultexp}.
The proof of Theorem \ref{theo_Tatteinte} is now complete.
\end{proof}

\section{Branching diffusion with interactions in a Brownian environment}\label{interactions}

We finish this paper with some interesting remarks on  branching diffusions with interactions in a Brownian environment. We decide to treat this case separately since  the competition mechanism $g$ may take negative and positive values and the techniques we use here are different from the rest of the paper. Our methodology are based on the theory of scale functions for diffusions. This allow us to provide a necessary and  sufficient condition  for extinction and moreover, the Laplace transform of  hitting times is computed explicitly in terms of a Ricatti equation. Such results seems complicated to obtain with the presence of jumps coming from the branching mechanism or the random environment

More general competition mechanisms were  considered by  Ba and Pardoux \cite{BPa} in the case when the branching mechanism is of the form $\psi(u)=\gamma^2 u^2$, for $u\ge 0$, see also Chapter 8 in the monograph of Pardoux \cite{Pardoux}. In this case, the CB-process with competition can be written as the unique strong solution of the following SDE
\[
Y_t=Y_0+\int_0^t h(Y_s)\ud s+\int_0^t\sqrt{2\gamma^2Y_s}\ud B^{(b)}_s,
\]
where $h$ is a continuous function satisfying $h(0)=0$ and such that 
\[
h(x+y)-h(x)\le K y, \qquad x, y\ge 0,
\]
for some positive constant $K$. According to Ba and Pardoux, the process $Y$ gets extinct in finite time if and only if
\[
\int_1^\infty\exp\left\{-\frac{1}{2}\int_1^u \frac{h(r)}{r}\ud r\right\}\ud u=\infty.
\]

Here, we focus on the  Feller diffusion case and general competition mechanism where  more explicit functionals of the process can be computed.
In this particular case, the process that we are interested on is defined by the unique strong solution of 
\begin{equation}\label{SDEBB}
Z_t=Z_0+b\int_0^tZ_s \ud s-\int_0^tg(Z_s)\ud s+\int_0^t\sqrt{2\gamma^2Z_s}\ud B^{(b)}_s+\int_0^t \sigma Z_{s}\ud  B^{(e)}_s,
\end{equation}
where $g$ is a real-valued continuous function satisfying the conditions in Proposition 1 in Palau and Pardo \cite{PP}. 

Our first main result provides a necessary and sufficient  condition for the process $Z$ defined by \eqref{SDEBB} to become extinct

\begin{theorem}
\label{theo_BB1}
Assume that $Z$ is the unique strong solution of  \eqref{SDEBB}, then
\begin{equation}
\label{equivalence}
\P_x\Big(T_0 < \infty\Big)=1 \qquad\textrm{if and only if}\qquad  \int^{\infty} \exp\left\{ 2\int_1^{u} \frac{g(z)-bz}{2\gamma^2 z+\sigma^2 z^2} \ud z \right\}\ud u=\infty.
\end{equation}
Moreover 
\[
\P_x\left(\lim_{t\to \infty} Z_t=\infty\right)=1-\P_x\Big(T_0 < \infty\Big).
\]
\end{theorem}
In particular, we may have the following situations
\begin{itemize}
\item[i)] If there exist $z_0>0$ and $w<b-\frac{\sigma^2}{2}$ such that for any $z\geq z_0$, $g(z)\leq wz$, then $\P_x(T_0<\infty)<1$.
An example of this situation is the cooperative case, that is to say when $g(z)$ is decreasing and $b>\frac{\sigma^2}{2}$.
\item[ii)] If there exist $z_0>0$ and $w>b-\frac{\sigma^2}{2}$ such that for any $z\geq z_0$, $g(z)\geq wz$, then $\P_x(T_0<\infty)=1$.\\
 An example of this situation are large competition mechanisms, that is to say for $g(z)\geq bz$ for any $z$ large enough. For instance, the latter holds for the so-called logistic case i.e. $g(z)=cz^2$. 
\end{itemize}

\begin{proof}[Proof of Theorem~\ref{theo_BB1}] 

We first observe  from Dubins-Schwarz Theorem,  that the law of $Z$ is equal to the law of the following diffusion
$$
\ud Y_t = (bY_t-g(Y_t))\ud t -\sqrt{2\gamma^2 Y_t+\sigma^2 Y_t^2} \ud W_t,
$$
where $W$ is a standard Brownian motion. Associated to $Y$, we introduce  for any $z\in \R$,
\[
\mathtt{b}(z):=g(z)-bz, \qquad \mathtt{d}(z):=\frac{1}{2}\Big(2\gamma^2z+\sigma^2 z^2\Big),
 \]
as well as the following functions related with the scale function of $Y$, for any $x,l\in \R_+$
\[
s(l)= \exp\left\{ \int_1^{l}\frac{\mathtt{b}(z)}{\mathtt{d}(z)} \ud z \right\}, 
 \quad S(l,x)=\int_l^x s(u) \ud u\quad \textrm{and}\quad
 \Sigma(l,x)=\int_l^x \left( \int_u^x \frac{1}{\mathtt{d}(\eta) s(\eta)}\ud \eta  \right) s(u)\ud u.
\]
Observe that for any $x\in \R_+$, 
\begin{equation}\label{scale_function}
S(0,x)=\int_0^x \exp\left\{2\int_1^u \frac{g(z)-bz}{2\gamma^2z+\sigma^2 z^2}\ud z \right\}\ud u.
\end{equation}
For simplicity, we denote $S(x)=S(0,x)$.

In order to prove the first statement of this proposition, we follow the approach of Chapter 15 in Karlin and Taylor \cite{KT1981} which ensures that the equivalence \eqref{equivalence} follows from the fact that $\lim_{l\to 0} \Sigma(l,x)$ is finite. Indeed, According to Lemma 15.6.3 in  \cite{KT1981}, the finiteness of $\lim_{l\to 0} \Sigma(l,x)$ for an $x>0$ implies the finiteness of $\lim_{l\to 0} S(l,x)=S(0,x)$ for all $x\geq 0$. Then Lemma 15.6.2 in  \cite{KT1981} guarantees that for any $y\geq x$,
$T_0 \wedge T_y < \infty, $ a.s., and Section 3 of Chapter 15 provides the following formulation
\begin{equation}
\label{eq_PT0}
\P_x(T_0 < T_y)=\frac{S(x)-S(y)}{S(0)-S(y)}.
\end{equation}
By making $y$ tend to $\infty$, we find the equivalence \eqref{equivalence} as required.

Hence let us show that $\lim_{l\to 0} \Sigma(l,x)$ is finite. In order to do so, we  fix $\varepsilon>0$ and $x\in (0,1)$ in such a way  that for any $z\leq x$, $|\mathtt{b}(z)|\le \varepsilon$. Therefore
\begin{equation}
\label{calcul_sigma}
\begin{aligned}
\Sigma(l,x) &= \int_l^x \left( \int_u^x \frac{1}{\mathtt{d}(\eta)}\exp\left\{ \int_\eta^{1}\frac{\mathtt{b}(z)}{\mathtt{d}(z)} \ud z  \right\}\ud \eta  \right) \exp\left\{-\int_u^{1}\frac{\mathtt{b}(z)}{\mathtt{d}(z)} \ud z \right\}\ud u \\
&\leq C_1(x) \int_l^x \left( \int_u^x \frac{1}{\mathtt{d}(\eta)}\exp\left\{  \int_\eta^{x}\frac{\varepsilon}{\mathtt{d}(z)} \ud z  \right\}\ud \eta  \right) \exp\left\{ \int_u^{x}\frac{\varepsilon}{\mathtt{d}(z)} \ud z \right\}\ud u\\
& \leq C_2(x) \int_l^x \left( \int_u^x \frac{1}{\mathtt{d}(\eta)}\left(\frac{1+\frac{\sigma^2}{2\gamma^2}\eta}{\eta}\right)^{\varepsilon/\gamma^2} \ud \eta \right) \left(\frac{1+\frac{\sigma^2}{2\gamma^2}u}{u}\right)^{\varepsilon/\gamma^2} \ud u, \\
\end{aligned}
\end{equation}
where $C_1(x)$ and $C_2(x)$ are positive constants that only depend on $x$.
Moreover, in  a neighbourhood of $0$, we have
$$
\frac{1}{\mathtt{d}(\eta)}\left(\frac{1+\frac{\sigma^2}{2\gamma^2}\eta}{\eta}\right)^{\varepsilon/\gamma^2} \underset{\eta\to 0}{\sim} \frac{1}{2\gamma^2} \frac{1}{\eta^{1+\varepsilon/\gamma^2}},
$$
which is not integrable at $0$. Hence,
$$
\int_u^x \frac{1}{\mathtt{d}(\eta)} \left(\frac{1+\frac{\sigma^2}{2\gamma^2}\eta}{\eta}\right)^{\varepsilon/\gamma^2} \ud \eta \underset{u \to 0}{\sim}  C_3(x) \frac{1}{u^{\varepsilon/\gamma^2}},
$$
where $C_3(x)$ is a positive constant that only depends on $x$.
This implies that the integrand on the right-hand side  of the last inequality in \eqref{calcul_sigma} is equivalent to $u^{-2\varepsilon/\gamma^2}$ which is integrable at $0$ as soon as $\varepsilon$ is chosen small enough. The latter implies that $\lim_{l\to 0} \Sigma(l,x)<\infty$ which completes the first statement of  this proposition.

In order to finish the proof, note that for any $y>x$,
$$
\P_x\left(\lim_{t\to \infty} Z(t)=\infty\right) \geq \P_x(T_y<T_0)=\frac{S(0)-S(x)}{S(y)-S(0)}.
$$
Since it  holds  for any $y \geq x$, we can take $y$ goes to $\infty$. By writing $S(\infty):=\lim_{y\to \infty} S(y)\in(0,\infty]$, we deduce
$$
\P_x\left(\lim_{t\to \infty} Z(t)=\infty\right) \geq \frac{S(0)-S(x)}{S(\infty)-S(0)},
$$
and the right-hand side  is equal to $1-\P_x(T_0<\infty)$ according to \eqref{eq_PT0}, whenever $S(\infty)$ is finite or not. This ends the proof.
\end{proof}

Our second result gives a formulation the Laplace transform of the first passage time 
\[
T_a=\inf\{t: Z_t\le a\}, \qquad \textrm{for} \quad a\ge 0,
\]  
by using the solution to the Ricatti equation described in the next Lemma and depending on the scale function $S$ defined by \eqref{scale_function}. The proof of Proposition ~\ref{theo_BB1} guarantees that $S$ is well-defined. Moreover, it is clear that the function $S:\R_+\to(0,S(\infty))$ is  continuous and bijective, and under condition~\eqref{equivalence}, $S(\infty)$ equals $\infty$. We denote by $\bar\varphi(x)$ the inverse of $S$ on $(0,S(\infty))$. Following similar arguments  to those provided in the proof of Lemma 2.1 in Lambert \cite{Lambert2005}, we deduce the following properties on the solution to the Ricatti equation that we are interested in.

\begin{lemma}
 \label{lemma_hqbis}
 For any $\lambda>0$, there exists a unique non-negative solution $\bar y_\lambda$ on $(0,S(\infty))$ to the equation
 \begin{equation*}
 y'=y^2-\lambda \bar r^2,
 \end{equation*}
where \[
\bar r(z)=\frac{ \bar \varphi'(z)}{\sqrt{\gamma^2\bar\varphi(z)+\frac{\sigma^2}{2} (\bar\varphi(z))^2}},
\]
 such that it vanishes at $S(\infty)$. Moreover, $\bar y_\lambda$ is positive on $(0,S(\infty))$, and for any $z$ sufficiently small or close to $S(\infty)$, $\bar y_\lambda(z) \leq \sqrt{\lambda}\bar r(z)$. In particular, $\bar y_\lambda$ is integrable at $0$ if $\gamma\neq 0$, and it decreases initially and ultimately.
 \end{lemma}

Our next result provides explicitly the Laplace transform of $T_a$ in terms of the function $\bar y_\lambda$.
 \begin{proposition}
 \label{theo_BB2}
Assume that $\gamma >0$. Then, for any $x\geq a \geq 0$, and for any $\lambda >0$,
 \begin{equation}
 \label{eq_BB2}
 \E_x\left[e^{-\lambda T_a} \right]=\exp\left\{-\int_{S(a)}^{S(x)}\bar y_\lambda(u) \ud u\right\}.
 \end{equation}
 Note that if \eqref{equivalence} is satisfied, then $T_a<\infty$ a.s.\\
 \end{proposition}

\begin{proof}
Let $x\geq a >0$, then $(Z_{t\wedge T_a}, t\ge 0)$, under $\P_x$, is a process with values in $[a,\infty)$. For any $y\geq a$, we define
\[
f_{\lambda,a}(y)=\exp\left\{-\int_{S(a)}^{S(x)} \bar y_\lambda(u) \ud u \right\}.
\]
A direct computation ensures that $f_{\lambda,a}$ is a $C^2$-function on $[a,\infty)$, bounded by $1$, $f_{\lambda,a}(a)=1$ and such that it solves
\begin{equation}
 \label{eq_EDO}
    \mathtt{d}(y)f''(y)-\mathtt{b}(y)f'(y)-\lambda f(y)=0.
 \end{equation}
 Applying It\^o Formula to the function $F(t,y)=e^{-\lambda t}f_{\lambda,a}(y)$ and the process $(Z_{t\wedge T_a}, t\ge 0)$, we obtain by means of \eqref{eq_EDO},
 \[
e^{-\lambda t}f_{\lambda,a}(Z_{t\wedge T_a})=f_{\lambda,a}(x)+\int_0^{t\wedge T_a} f_{\lambda,a}'(Z_s)\sqrt{2\gamma^2 Z_s}\ud B^{(b)}_s+\sigma \int_0^{t\wedge T_a}f_{\lambda,a}'(Z_s)Z_s\ud B^{(e)}_s. 
 \]
 We then use a sequence of stopping time $(T_n, n\ge 1)$ that reduces the two local martingales of the right-hand side and from the optimal stopping theorem, we obtain for any $n\ge 1$
 \[
 \E_x\left[ e^{-\lambda T_n \wedge T_a} f_{\lambda,a}(Z_{T_n\wedge T_a}) \right]=f_{\lambda,a}(x).
 \]
 Letting $n$ go to $\infty$ gives \eqref{eq_BB2} for any $x\geq a>0$. We finally let $a$ go to $0$ to deduce the result for $a=0$ and conclude the proof.
 \end{proof}

\noindent \textbf{Acknowledgements.} Both authors  acknowledge support from  the Royal Society and CONACyT-MEXICO. This work was concluded whilst JCP was on sabbatical leave holding a David Parkin Visiting Professorship  at the University of Bath, he gratefully acknowledges the kind hospitality of the Department and University


\begin{thebibliography}{99}







\bibitem{BPa} {\sc Ba, M. and Pardoux, E. } {Branching processes with interaction and a generalized Ray-Knight theorem.}  {\it  Ann. Inst. Henri Poincar\'e Probab. Stat.},  51, no. 4, 1290--1313, (2015).
\bibitem{BPS1}{\sc Bansaye, V., Pardo, J.C. and Smadi, C.} {Extinction rate of continuous state branching processes in critical L\'evy environments.} {\it Preprint}, (2019).  {\tt Arxiv:1903.06058}. 

\bibitem{BFF} {\sc Berestycki, J., Fittipaldi, M.C. and Fontbona, J.}  {Ray-Knight representation of flows of branching processes with competition by pruning of L\'evy trees.}, {\it  Probab. Theory  Related Fields}, 172, no. 3-4, 725--788, (2018). 

\bibitem{DK} {\sc D\"oring, L. and Kyrpianou, A. } (2018). {\rm  Entrance and exit at infinity for stable jump diffusions,} {\it Preprint}. Available at {\tt arXiv:1802.01672.}


\bibitem{DL} {\sc Dawson, D.A., Li, Z.} { Stochastic equations, flows and measure-valued processes.} {\it Ann. Probab.} 40, 813--857, (2012).

\bibitem{DFM}{\sc Duhalde, X., Foucart, C. and  Ma, C.} { On the hitting times of continuous-state branching processes with immigration.} {\it Stochastic Process. Appl.} 124, 4182--4201, (2014) .

\bibitem{EHS} {\sc S.E. Evans, A. Hening and S. Schreiber:} Protected polymorphisms and evolutionary stability of patch-selection strategies in stochastic environments. {\it J. Math. Biol.},  {\bf 71},  325--359, (2015). 

\bibitem{Fou}{\sc  Foucart, C.} { Continuous state branching processes with competition, duality and reflection at infinity} {\it Electron. J. Probab.}, 24, no. 33, 1--38,  (2019). 

\bibitem{GPP} {\sc González-Casanova, A., Pardo, J.C. and Perez., J.L.} {Branching processes with interactions: the subcritical cooperative regime.} {\it Preprint} (2017) {\tt arXiv:1704.04203.}

\bibitem{IW89} \sc Ikeda, N., Watanabe, S. {\it Stochastic Differential Equations and Diffusion Processes.} \rm North-Holland Publishing Company, New York, (1989).


\bibitem{Kallenberg} {\sc Kallenberg, O.} {\it Foundations of Modern Probability}.  {\rm Probability and Its Applications,  Springer-Verlag, New York,} (1997).

\bibitem{KRM}{\sc Keller-Ressel, M. and  Mijatovi\'c, A.} { On the limit distributions of continuous-state branching processes with immigration.} {\it Stochastic Process. Appl.} 122, 2329--2345, (2012) .

\bibitem{KT1981}{\sc Karlin, S. and Taylor, H. E.} {\it A second course in stochastic processes.} \rm  Academic Press, London, (1981).

\bibitem{Lambert2005}{\sc Lambert, A.}  {\rm The branching process with logistic growth}, {\it Ann. Appl. Probab.}, 15, no. 2,  1506--1535,  (2005).



\bibitem{La}{\sc Lamperti, J. } {\rm Continuous state branching processes}, {\it Bull. Amer. Math. Soc.},  73, 382--386, (1967).

\bibitem{LP1}{\sc Leman, H. and Pardo, J.C. } {\rm Extinction and coming down from infinity of CB-processes with competition in a L\'evy environment}, {\it Preprint} {\tt arXiv:1801.04501v3} (2019).


\bibitem{li2010}{\sc Li, Z.} {\it Measure-valued branching Markov processes.} {\rm Probability and Its Applications,  Springer-Verlag, Berlin Heidelberg} (2011).
 (2010).

\bibitem{Ma2015} {\sc Ma, R.} {\it Lamperti transformation for continuous-state branching processes with competition and applications},  Statist. Probab. Lett. 107, 11--17, (2015).

\bibitem{PP0} {\sc Palau, S. and Pardo, J.C.} {\rm Continuous state branching processes in random environment: The Brownian case.} {\it   Stochastic Process. Appl., } 127, no.3, 957--994,  (2017). 


\bibitem{PP} {\sc Palau, S. and Pardo, J.C.} {\rm Branching processes in a L\'evy random environment.} {\it  Acta Appl. Math., } 153, no.1, 55--79,  (2018). 

\bibitem{Pardoux} {\sc Pardoux, E. }
{\it Probabilistic models of population evolution. Scaling limits, genealogies and interactions.} Mathematical Biosciences Institute Lecture Series. Stochastics in Biological Systems, 1.6. Springer ; MBI Mathematical Biosciences Institute, Ohio State University, Columbus, OH, (2016).


\bibitem{le2014processus} {\sc Le, V. } {\it Processus de branchement avec interaction.} {PhD Thesis, Univ. Aix-Marseille} (2014).
 
 \bibitem{vi2015height} {\sc Le, V. and Pardoux, E.} {\rm Height and the total mass of the forest of genealogical trees of a large population with general competition.} {\it ESAIM: Probability and Statistics}, 19, 172--193 (2015).
 
 \bibitem{he2016continuous} {\sc He, H. and Li, Z. and Xu, W.} {\rm Continuous-state branching processes in L{\'e}vy random environments} {\it J. Theor. Probab.}, 31, no. 4, 1952--1974,  (2018). 
 
\bibitem{RevuzYor} \sc Revuz, D. and Yor, M. {\it Continuous martingales and Brownian motion.} \rm Springer-Verlag,  Berlin Heidelberg, (1999). 

\bibitem{Sa} \sc Sato, K. {\it L{\'e}vy processes and infinitely divisible distributions}.  {\rm Cambridge University Press, Cambridge,} (1999).

\end{thebibliography}
\end{document}